\documentclass[12pt]{amsart}

\usepackage[margin=1in]{geometry}
\usepackage{amsmath, amssymb, amsthm, amssymb}
\usepackage{amsfonts}
\usepackage{amsmath}
\usepackage{graphicx}
\usepackage[frame,ps,matrix,arrow,curve,rotate]{xy}
\usepackage{enumerate}
\usepackage{tikz}
\usepackage{hyperref}
\usepackage{siunitx}
\usepackage{enumitem}
\usepackage{xspace}

\usepackage{blkarray}
\usepackage{tikz-cd}

\newtheorem{theorem}{Theorem}[section]
\newtheorem{lemma}[theorem]{Lemma}

\newtheorem{proposition}[theorem]{Proposition}
\newcounter{claims}[theorem]
\newtheorem{claim}[claims]{Claim}

\theoremstyle{definition}

\theoremstyle{remark}

\usepackage{etoolbox}
\newcommand{\define}[4]{\expandafter#1\csname#3#4\endcsname{#2{#4}}}

\forcsvlist{\define{\DeclareMathOperator}{}{}}{im,coker,codim,rank,diam,vol,Tr,Nm,Ind,Res,Cor,Sym,Alt,Ob,Ext,Hom,Mor,End,Aut,Gal,lcm,gcf,sign,Spec,Maxspec,Nil,JRad,dom,ext,Th,Bilin,tr}
\renewcommand{\int}{\text{int}}

\forcsvlist{\define{\newcommand}{\mathrm}{}}{id,op,GL,SL,PSL,PGL,SO,SU,U,Mat}

\forcsvlist{\define{\newcommand}{\mathsf}{}}{Set,Grp,Ab,CRing,CAlg,Mod,Vect,Cat,Top,Alg,Rep,Ring,Field,Rig}

\forcsvlist{\define{\newcommand}{\mathbb}{}}{N,Z,Q,R,C,F}
\newcommand{\bR}{\mathbb{R}}

\newcommand{\bN}{\mathbb{N}}

\newcommand{\bZ}{\mathbb{Z}}
\newcommand{\bS}{\mathbb{S}}

\newcommand{\mc}{\mathcal}

\newcommand{\cS}{\mathcal{S}}

\newcommand{\del}{\partial}
\newcommand{\isom}{\cong}
\newcommand{\conj}{\simeq}
\newcommand{\res}{\!\!\upharpoonright}

\newcommand{\abs}[1]{\lvert#1\rvert}

\newcommand{\set}[1]{\{#1\}}

\renewcommand{\to}{\rightarrow}
\renewcommand{\conj}{\approx}
\newcommand{\homeo}{\isom}

\renewcommand{\geq}{\geqslant}

\renewcommand{\leq}{\leqslant}
\renewcommand{\emptyset}{\varnothing}
\renewcommand{\epsilon}{\varepsilon}

\newcommand{\red}{\color{red}}

\let\ol\overline

\renewcommand{\H}{\mathcal{H}}
\usepackage{pgfplots}
\pgfplotsset{compat=1.15}
\usepackage{mathrsfs}
\usetikzlibrary{arrows}
\definecolor{qqzzff}{rgb}{0.,0.6,1.}
\definecolor{ududff}{rgb}{0.30196078431372547,0.30196078431372547,1.}
\definecolor{xdxdff}{rgb}{0.49019607843137253,0.49019607843137253,1.}
\definecolor{uuuuuu}{rgb}{0.26666666666666666,0.26666666666666666,0.26666666666666666}
\definecolor{zzttqq}{rgb}{0.6,0.2,0.}
\definecolor{cqcqcq}{rgb}{0.7529411764705882,0.7529411764705882,0.7529411764705882}
\definecolor{ttqqqq}{rgb}{0.2,0.,0.}

\begin{document}
\title{The homeomorphisms of the Sierpi\'nski carpet are not classifiable by countable structures}
\author{Dhruv Kulshreshtha and Aristotelis Panagiotopoulos}


\thanks{This research was supported by the
NSF Grant DMS-2154258: ``Dynamics Beyond Turbulence and Obstructions to Classification''. 
}

\maketitle
\markboth{Dhruv Kulshreshtha and Aristotelis Panagiotopoulos}{The homeomorphisms of the Sierpi\'nski carpet are not classifiable by countable structures} 

\begin{abstract}
We show that  the homeomorphisms of the Sierpi\'nski carpet are not classifiable, up to conjugacy, using isomorphism types of countable structures as invariants.
\end{abstract}

\section{Introduction}

Let $K$ be a compact metrizable space. A  homeomorphism of $K$ is any continuous bijection $\varphi\colon K\to K$. The collection $\H(K)$ of all homeomorphisms of $K$ forms a group under composition. The elements of $\mathcal{H}(K)$ constitute all possible ``symmetries" of the topological space $K$.  Two such symmetries may represent the exact same dynamics on $K$, just in a different ``coordinate system". Consider for example the reflections $\chi, \psi \in\H(T)$ of the  unit circle  $T=\{z\in\mathbb{C}\colon |z|=1\}$ about the $x$ and $y$ axis, respectively, and notice that, up to rotation of $T$ by   $\pi/2$, they induce the  the same dynamics on $T$. More generally, two homeomorphisms $\rho, \rho'$ of $K$ are ``essentially the same"  if they are {\bf conjugate}  $\rho\conj \rho'$, i.e., if 
\[\rho=\sigma \circ \rho' \circ \sigma^{-1} \text{ for some } \sigma\in \H(K).\]

Classifying all possible truly distinct symmetries of $K$ is a fundamental problem in topological dynamics; but one that is rarely feasible using  ``concrete" invariants. Indeed the best one can hope for 
these type of classification problems
is ``classification by countable structures": finding an assignment $\rho\mapsto \mathcal{A}(\rho)$  from  $\H(K)$  to some space  of countable structures (graphs, groups, rings, etc.) so that  $\rho\conj \rho'$ holds if and only if the structures  $\mathcal{A}(\rho)$  and $\mathcal{A}(\rho')$  are isomorphic.  It turns out, the interplay between the topological dimension of $K$ and  classifiability of $\big(\H(K), \conj\hspace{-3pt}\big)$ by countable structures is quite  subtle. For example:
\begin{enumerate}
\item[(1)] $\big(\H(K),\conj\hspace{-3pt}\big)$ is classifiable by countable structures, if  $K$ is  $0$--dimensional;
\item[(2)]   $\big(\H([0,1]),\conj\hspace{-3pt}\big)$ is classifiable by countable structures;
\item[(3)] $\big(\H([0,1]^2),\conj\hspace{-3pt}\big)$ is not classifiable by countable structures.
\end{enumerate}

The idea behind the positive results (1) and (2) above is simple.
Indeed, when $K$ is $0$--dimensional, one can associate to each $\rho\in\H(K)$ the countable structure $\mathcal{B}_{\rho}:=(B,f_{\rho})$, where $B$ is the Boolean algebra of all clopen subsets of $K$ and $f_{\rho}$ is the automorphism of $B$ induced by the action of $\rho$ on clopen sets. By  Stone duality this assignment is a complete classification. Similarly, for each $\rho\in\H([0,1])$ we can find a certain  countable collection $\mathcal{I}_{\rho}$ of disjoint open intervals $(a,b)\subseteq [0,1]$ so that on each $I\in\mathcal{I}_{\rho}$ the homeomorphism $\rho$ is  everywhere increasing, or everywhere deceasing, or everywhere constant. Endowing the countable set $\mathcal{I}_{\rho}$ with a natural ordering and with predicates  recording the aforementioned behavior of  $\rho$ on each $I\in\mathcal{I}_{\rho}$ provides a  classification of  $\big(\H([0,1]),\conj\big)$ by countable structures; see \cite[Section 4.2]{Hjorth2010}, for details.
The negative anti-classification result (3) above is due to Hjorth  \cite[Theorem 4.17]{Hjorth2010} and it indicates that  isomorphism types of countable structures are too simple of invariants to capture all different behaviors  present in $\H([0,1]^2)/{\conj}$.  To make this statement precise, we recall some nomenclature  from invariant descriptive set theory.  

A {\bf classification problem} is a pair $(X,E)$ where $X$ is a Polish space and $E$ is an equivalence relation on $X$.  We say that the classification problem $(X,E)$ is {\bf Borel reducible} to $(Y,F)$,  and  write  $(X,E)\leq_B (Y,F)$,   if   $xEx' \iff f(x)F f(x')$ holds for some Borel map $f\colon X\to Y$.  The Borel assumption here  indicates that there is a ``constructive" way to use the abstract objects in $Y/F$ as invariants for classifying $X/E$.  Such assumption is pertinent, as both quotients $X/E$ and $Y/F$ will typically have the size of the continuum, in which case the axiom of choice guarantees the existence of a (non-constructive) way of ``classifying"  $X/E$ using $Y/F$. 
Within the Borel reduction hierarchy ``classifiability by countable structures" forms a well--defined complexity class whose most complex element is graph isomorphism; see \cite{Gao,Hjorth2010}. In short, we may define a  classification problem $(X,E)$ to be {\bf classifiable by countable structures} if $(X,E)\leq_B (\mathrm{Graphs}(\mathbb{N}), \simeq_{\mathrm{iso}})$. Here,
 $\mathrm{Graphs}(\mathbb{N})$  is identified with the closed subset of the Cantor space $\{0,1\}^{\mathbb{N}\times \mathbb{N}}$ consisting of all $x$ in the latter with $x(n,n)=0$ and $x(n,m)=x(m,n)$ for all $n,m\in\mathbb{N}$. The isomorphism relation  $x\simeq_{\mathrm{iso}}y$ holds if there is a bijection $p\colon \mathbb{N}\to \mathbb{N}$ so that $x(n,m)\iff y(p(n),p(m))$. 
  Since the compact--open topology on $\H(K)$ is Polish when the compact set $K$ is  metrizable, we may view  $\big(\H(K),\conj \hspace{-3pt}\big)$ as a classification problem in the  aforementioned formal sense. 
  
  In view of  (1)--(3) above, one wonders how sensitive is the complexity of $\big(\H(K),\conj\hspace{-3pt}\big)$  to the dimension of $K$. Are there examples of 1--dimensional spaces $K$ with $\big(\H(K),\conj\hspace{-3pt}\big)$  not being classifiable by countable structures? And if so, can we take $K$ to satisfy further properties (planar, locally connected)? The answer turns out to be positive. Recall that the Sierpi\'nski carpet $\cS$  is a one--dimensional planar continuum which can be attained by subtracting from the two--dimensional sphere $S^2$ the interiors $\mathrm{int}(D_n)$ of any sequence $(D_n\colon n\in \mathbb{N})$ of  pairwise disjoint closed balls $D_n \subseteq S^2$ whose union is dense in $S^2$; see \cite{Whyburn1958}.
  
\begin{theorem}\label{T:Main}
$\big(\H(\mathcal{S}),\conj\hspace{-3pt}\big)$ is not classifiable by countable structures.
\end{theorem}

The proof  follows the general lines of Hjorth's proof of  \cite[Theorem 4.17]{Hjorth2010}. That being said, one needs to improvise significantly in order to remedy the extra rigidity of $\mathcal{S}$ which, unlike $[0,1]^2$, lacks any affine structure. In fact, a priori it is not at all clear that Theorem \ref{T:Main} should be true. Indeed, actions of non-archimedean Polish groups induce orbit equivalence relations which are always classifiable by countable structures \cite[Theorem 2.39]{Hjorth2010} and $\H(\mathcal{S})$ is, in a certain sense, close to being non-archimedean; see \cite[Theorem 1.2]{Brechner}. 

\begin{figure}[ht!]
\centering
\includegraphics[scale=0.09]{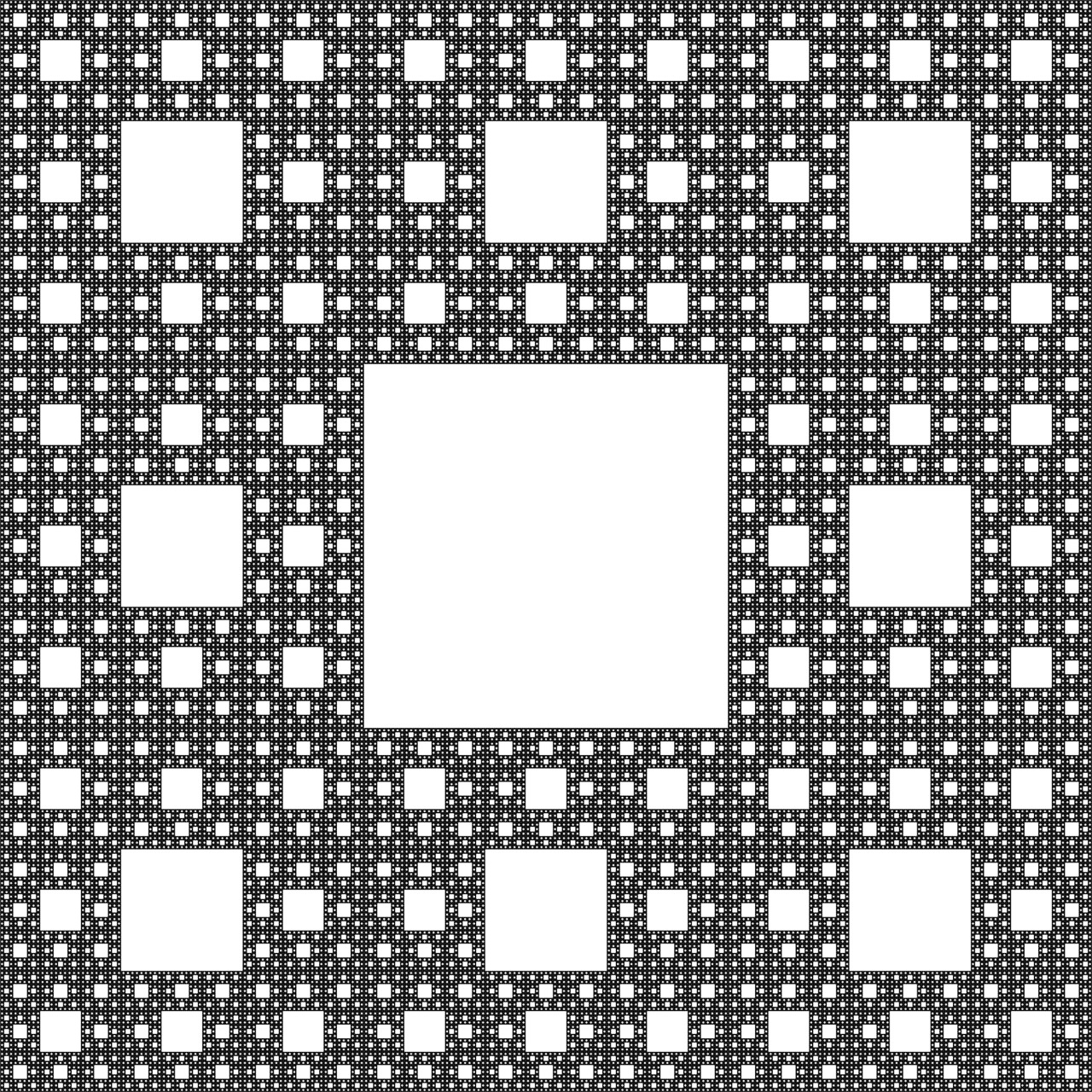}
\caption{The  Sierpinski Carpet $\mathcal{S}$}\label{F:intro1}
\end{figure}

Finally, we should remark that the interplay between topological dimension and various other classification problem has been recently gaining  some visibility \cite{CamDM, Kru,  Dudak, Basso}. For example, consider the problem of classifying  all compact metrizable spaces of dimension at most $n$ up to homeomorphism. While we know  this problem to be classifiable by countable structures \cite{CamGao} for $n=0$, and universal for Polish group actions \cite{Zielinski} for $n=\infty$, the lack of  a ``higher dimensional" version of  Hjorth's  turbulence theory makes the task of  comparing the complexity between intermediate dimensions highly intractable. Theorem \ref{T:Main} is  a biproduct---and in some sense  a first step---in our premature attempts to develop turbulence-like obstructions to classification by actions of homeomorphism groups $\H(K)$, when the dimension of $K$ is at most one.


\section{Preliminaries}\label{SS:Preliminaries}

By a closed ball we mean any subset $B$ of the plane $\mathbb{R}^2$ that is homeomorphic to the closed unit ball $\{(x,y)\in \mathbb{R}^2 \colon |x|^2+|y|^2\leq 1 \}$.
Let $C\subseteq \mathbb{R}^2$ be  closed subset of the plane and let $\partial C=C\setminus\mathrm{int}(C)$ be its boundary. By a  {\bf deballing recipe for $C$} we mean any countable collection  $\mathcal{B}$ of  pairwise disjoint closed balls $B \subseteq \mathbb{R}^2$  so that  $\bigcup_{B\in\mathcal{B}} B$ is dense in $\mathrm{int}(C)$ and  does not intersect $\partial C$.  We denote by $\mathcal{B}(C)=C\setminus \bigcup_{B\in\mathcal{B}} \mathrm{int}(B)$  the associated {\bf deballing} of $C$. The Sierpinski Carpet $\mathcal{S}$ in Figure \ref{F:intro1}, for example, is of the form $\mathcal{B}(C)$
for an appropriate deballing recipe $\mathcal{B}$ of $C = [0,1]^2$.

It turns out that if $C$ is any closed ball and $\mathcal{B}$ is any   deballing recipe for $C$, then  $\mathcal{B}(C)$ is homeomorphic to the Sierpinski Carpet  $\mathcal{S}$ from Figure \ref{F:intro1}. In fact, we have that:

\begin{theorem}[\cite{Whyburn1958}]\label{T:whyburn} If $\mathcal{B}_0$, $\mathcal{B}_1$ are  deballing recipes for the closed balls $C_0$ and $C_1$, respectively, then any homeomorphism $h\colon \partial C_0\to \partial C_1$ extends to a homeomorphism 
\[\widehat{h}\colon \mathcal{B}_0(C_0)\to \mathcal{B}_1(C_1).\]
\end{theorem}

Notice that if a closed set $C\subseteq \mathbb{R}^2$ is a union $C=\bigcup_i C_i$ of  a family
$\{C_i\colon i\in  I\}$
 of closed sets with $C_i \cap C_j\subseteq \partial C_i \cap \partial C_j$ for all $i\neq j$ and,  for each $i\in I$, $\mathcal{B}_i$ is a deballing recipe for $C_i$ with $\bigcup \mathcal{B}_i\subseteq C_i$, then $\bigcup_i\mathcal{B}_i$ is a deballing recipe for $C$.
Similarly, if $C,D\subseteq \mathbb{R}^2$ are closed and $h\colon \mathrm{int}(C)\to \mathrm{int}(D)$ is a homeomorphism,   
 any deballing recipe $\mathcal{B}$ for $D$ {\bf pulls back} to a deballing recipe $h^*\mathcal{B}=\{h^{-1}(B)\colon B\in \mathcal{B}\}$ for $C$.

Let $K:=[0,1]\times \mathbb{R}$ be the infinite strip and let $\widehat{K}:=K\cup\{-\infty,+\infty\}$ be the two point compactification of $K$. A basis of open neighborhoods around $-\infty$ in  $\widehat{K}$ consists of all sets of the form $\big([0,1]\times (-\infty,r)\big)\cup\{-\infty\}$, where $r\in\mathbb{R}$; and a  basis of open neighborhoods around $+\infty$ consists of all sets of the form $\big([0,1]\times (r,+\infty)\big)\cup\{+\infty\}$, where $r\in\mathbb{R}$. Clearly  $\widehat{K}$ is homeomorphic to a closed ball. It will be convenient for what follows to fix some homeomorphism between the right-half disc and  $\widehat{K}$: 
\begin{equation}\label{EQ:compactification}
\kappa\colon  \{(x,y)\in \mathbb{R}^2\colon |x|^2+|y|^2\leq 1 \text{ and } x\geq 0\} \to  \widehat{K}
\end{equation} 
so that  $\kappa(0,-1)=-\infty$,  $\kappa(0,+1)=+\infty$, and  $\kappa(x,0)=(x,0)$ for all $x\in [0,1]$.







\section{The outline of the proof}

The  proof of Theorem \ref{T:Main}  follows the general lines  of \cite[Theorem 4.17]{Hjorth2010}. More precisely, consider the Polish group  $G := \prod_{\bN}\bZ$, with the product group structure and topology, and set $\widehat{G}:= \set{g\in G \mid  \lim_{n\to\infty} g(n)/(n+1) = 0}$. Let also $E_{\widehat{G}}$ be the orbit equivalence relation induced by the action of $\widehat{G}$ on $G$  by left translation: for $g,g'\in G$,
\[gE_{\widehat{G}} g' \iff g-g'\in  \widehat{G}.\]
As in  \cite[Lemma 4.16]{Hjorth2010},  it follows that $(G,E_{\widehat{G}})$ is not classifiable by countable structures. Hence, in order to prove Theorem \ref{T:Main}, it suffices to establish that   $(G,E_{\widehat{G}})\leq_B
\big(\H(\mc{S}),\conj\hspace{-3pt}\big)$.

\begin{proposition}\label{Prop:reduction}
There is  a Borel  $\rho : G \to \H(\cS)$, so that for all $g,g' \in G$,  we have 
\begin{equation}
gE_{\widehat{G}} g'  \; \iff \; \rho(g) \conj \rho(g'). 
\end{equation}
\end{proposition}

For the proof of Proposition \ref{Prop:reduction} we  will have to rely on the following two lemmas. The proofs of these lemmas will be provided in the Sections \ref{SL:main1} and \ref{SL:main2}, respectively.

\smallskip{}

\subsection{Two Lemmas}
Let $I:=[0,1]^2$ be the unit square and let $\mathcal{B}(I)$ be the Sierpinski carpet resulting from some deballing recipe on $I$; see, e.g., Figure \ref{F:intro1}. Our first lemma states that $\mathcal{B}(I)$ admits some homeomorphism $\pi$ which fixes $\partial I$ pointwise, while the trajectory of  every  $z\in\mathcal{B}(I)\setminus \partial I$ via $\pi$ accumulates to the top and bottom part of the boundary $\partial I$ in a prescribed way. In view of Theorem \ref{T:whyburn}, this can be formulated as follows:

\begin{lemma}\label{L:main1}
There exists a deballing recipe $\mathcal{B}_I$ for $I$ and some $\pi\in\mathcal{H}\big(\mathcal{B}_I(I)\big)$, so that:
\begin{enumerate}
\item if $z\in   \partial I$, then $\pi(z)=z$;
\item if $z=(x,y)\in\mathcal{B}_I(I) \setminus \partial I$, then we have that 
\[\lim_{n\to-\infty}\pi^n(z)=(x,0)\; \text{ and  }\; \lim_{n\to\infty}\pi^n(z)=(x,1).\]
\end{enumerate}
\end{lemma}

\smallskip{}

The next lemma allows us to compensate for the lack of affine structure of the Sierpinski carpet.
Let $L:=[0,1]\times\mathbb{R}$, so that  $\partial L=\{0,1\} \times \mathbb{R}$. 
For every   $k,\ell\in \mathbb{Z}$, let
\begin{equation}
h_{k,\ell}^{\partial}\colon \partial L\to \partial L \quad \text{ with } \quad h^{\partial}_{k,\ell}(0,y)=(0,y+k) \text{ and } h^{\partial}_{k,\ell}(1,y)=(1,y+\ell).
\end{equation}
Using the affine structure on $L$ one can   extend   $h_{k,\ell}^{\partial}$ to the homeomorphism of $L$
\begin{equation}
h_{k,\ell}\colon L\to L \quad \text{ with } \quad h_{k,\ell}(x,y)=\big(x, (1-x)(y+k) + x (y+\ell)\big).
\end{equation}
Similarly, using Theorem \ref{T:whyburn} and the compactification $\kappa$ from (\ref{EQ:compactification}), one may  extend $h_{k,\ell}^{\partial}$ to a homeomorphism $h^{\mathcal{B}}_{k,\ell}\colon \mathcal{B}(L) \to  \mathcal{B}(L)$ for any deballing recipe $\mathcal{B}$ for $L$.  
However, Theorem \ref{T:whyburn} does not provide enough control on the extension, and the maps $h^{\mathcal{B}}_{k,\ell}$ and $h_{k,\ell}\upharpoonright  \mathcal{B}(L)$ may grow unboundedly far as we keep sampling from $(x,y)\in \mathcal{B}(L)$ with $y\to\pm\infty$. 
The next lemma  shows that one may always choose an appropriate deballing recipe $\mathcal{B}$  for $L$ so that,  for all $k,\ell\in\mathbb{Z}$, $h_{k,\ell}^{\partial}$ admits an  extension $h^{\mathcal{B}}_{k,\ell}$ of bounded distortion relative to $h_{k,\ell}\upharpoonright  \mathcal{B}(L)$.

\begin{lemma}\label{L:main2}
There exists a deballing recipe $\mathcal{B}_L$ for $L$ so that for every $k,\ell\in\mathbb{Z}$ there exists a homeomorphism $h^{\mathcal{B}}_{k,\ell}\colon \mathcal{B}_L(L) \to  \mathcal{B}_L(L)$ which  extends $h^{\partial}_{k,\ell}$,
so that for all $(x,y)\in L$, if we set $(\tilde{x},\tilde{y}):=h^{\mathcal{B}}_{k,\ell}(x,y)$, then we have that
\[   \lfloor y \rfloor +\min\{k,\ell\}-1 \;  \leq \;  \tilde{y} \;  \leq  \; (\lfloor y \rfloor+1) +\max\{k,\ell\}+1,\]
where $\lfloor y \rfloor$  is the unique integer with  $0\leq y-\lfloor y \rfloor <1$.
\end{lemma}

\subsection{The Borel reduction} Using the above two lemmas, can now  prove Proposition \ref{Prop:reduction}. Let $K:=[0,1]\times \mathbb{R}$ be the infinite strip and decompose $K$ into a union
\begin{gather*}
  K=\big(\bigcup_{n} L_n\big)\cup\big(\bigcup_{n,k} I_{n,k}\big), \; \text{ with } \\
  L_n:=[\frac{2}{4^{n+1}},\frac{1}{4^{n}}]\times \mathbb{R}  \quad \text{ and } \quad  I_{n,k}:=
[\frac{1}{4^{n+1}}, \frac{2}{4^{n+1}}]\times [\frac{k-1}{n+1}, \frac{k}{n+1}],
\end{gather*}
where $n\in \mathbb{N}$ ranges over the natural numbers   and  $k\in \mathbb{Z}$ over the integers; see Figure \ref{F:2}.


\begin{figure}[ht!]
\begin{tikzpicture}[line cap=round,line join=round,>=triangle 45,x=12.0cm,y=4.0cm,scale=0.9]
\begin{axis}[
x=12.0cm,y=4.0cm,
axis lines=middle,
xmin=-0.15,
xmax=1.15,
ymin=-1.15,
ymax=1.15,
xtick={-0.0,...,1.0},
ytick={-1.0,0.0,...,1.0},]
\clip(-0.15,-1.15) rectangle (1.15,1.15);
\fill[line width=2.pt,color=ududff,fill=ududff,fill opacity=0.10000000149011612] (0.25,0.) -- (0.25,0.5) -- (0.5,0.5) -- (0.5,0.) -- cycle;
\fill[line width=2.pt,color=ududff,fill=ududff,fill opacity=0.10000000149011612] (0.25,0.5) -- (0.25,1.) -- (0.5,1.) -- (0.5,0.5) -- cycle;
\fill[line width=2.pt,color=ududff,fill=ududff,fill opacity=0.10000000149011612] (0.25,0.) -- (0.5,0.) -- (0.5,-0.5) -- (0.25,-0.5) -- cycle;
\fill[line width=2.pt,color=ududff,fill=ududff,fill opacity=0.10000000149011612] (0.25,-1.) -- (0.25,-0.5) -- (0.5,-0.5) -- (0.5,-1.) -- cycle;
\fill[line width=2.pt,color=ududff,fill=ududff,fill opacity=0.10000000149011612] (0.06,-1.) -- (0.06,-0.75) -- (0.13,-0.75) -- (0.13,-1.) -- cycle;
\fill[line width=2.pt,color=ududff,fill=ududff,fill opacity=0.10000000149011612] (0.06,-0.75) -- (0.06,-0.5) -- (0.13,-0.5) -- (0.13,-0.75) -- cycle;
\fill[line width=2.pt,color=ududff,fill=ududff,fill opacity=0.10000000149011612] (0.06,-0.5) -- (0.06,-0.25) -- (0.13,-0.25) -- (0.13,-0.5) -- cycle;
\fill[line width=2.pt,color=ududff,fill=ududff,fill opacity=0.10000000149011612] (0.06,-0.25) -- (0.06,0.) -- (0.13,0.) -- (0.13,-0.25) -- cycle;
\fill[line width=2.pt,color=ududff,fill=ududff,fill opacity=0.10000000149011612] (0.06,0.) -- (0.06,0.25) -- (0.13,0.25) -- (0.13,0.) -- cycle;
\fill[line width=2.pt,color=ududff,fill=ududff,fill opacity=0.10000000149011612] (0.06,0.25) -- (0.06,0.5) -- (0.13,0.5) -- (0.13,0.25) -- cycle;
\fill[line width=2.pt,color=ududff,fill=ududff,fill opacity=0.10000000149011612] (0.06,0.5) -- (0.06,0.75) -- (0.13,0.75) -- (0.13,0.5) -- cycle;
\fill[line width=2.pt,color=ududff,fill=ududff,fill opacity=0.10000000149011612] (0.06,0.75) -- (0.06,1.) -- (0.13,1.) -- (0.13,0.75) -- cycle;
\fill[line width=2.pt,color=ududff,fill=ududff,fill opacity=0.10000000149011612] (0.015,1.) -- (0.03,1.) -- (0.03,0.875) -- (0.015,0.875) -- cycle;
\fill[line width=2.pt,color=ududff,fill=ududff,fill opacity=0.10000000149011612] (0.015,0.875) -- (0.015,0.75) -- (0.03,0.75) -- (0.03,0.875) -- cycle;
\fill[line width=2.pt,color=ududff,fill=ududff,fill opacity=0.10000000149011612] (0.015,0.75) -- (0.015,0.625) -- (0.03,0.625) -- (0.03,0.75) -- cycle;
\fill[line width=2.pt,color=ududff,fill=ududff,fill opacity=0.10000000149011612] (0.015,0.625) -- (0.015,0.5) -- (0.03,0.5) -- (0.03,0.625) -- cycle;
\fill[line width=2.pt,color=ududff,fill=ududff,fill opacity=0.10000000149011612] (0.015,0.5) -- (0.015,0.375) -- (0.03,0.375) -- (0.03,0.5) -- cycle;
\fill[line width=2.pt,color=ududff,fill=ududff,fill opacity=0.10000000149011612] (0.015,0.375) -- (0.015,0.25) -- (0.03,0.25) -- (0.03,0.375) -- cycle;
\fill[line width=2.pt,color=ududff,fill=ududff,fill opacity=0.10000000149011612] (0.015,0.25) -- (0.015,0.125) -- (0.03,0.125) -- (0.03,0.25) -- cycle;
\fill[line width=2.pt,color=ududff,fill=ududff,fill opacity=0.10000000149011612] (0.015,0.125) -- (0.015,0.) -- (0.03,0.) -- (0.03,0.125) -- cycle;
\fill[line width=2.pt,color=ududff,fill=ududff,fill opacity=0.10000000149011612] (0.015,-0.125) -- (0.015,0.) -- (0.03,0.) -- (0.03,-0.125) -- cycle;
\fill[line width=2.pt,color=ududff,fill=ududff,fill opacity=0.10000000149011612] (0.015,-0.25) -- (0.015,-0.125) -- (0.03,-0.125) -- (0.03,-0.25) -- cycle;
\fill[line width=2.pt,color=ududff,fill=ududff,fill opacity=0.10000000149011612] (0.015,-0.25) -- (0.015,-0.375) -- (0.03,-0.375) -- (0.03,-0.25) -- cycle;
\fill[line width=2.pt,color=ududff,fill=ududff,fill opacity=0.10000000149011612] (0.015,-0.375) -- (0.03,-0.375) -- (0.03,-0.5) -- (0.015,-0.5) -- cycle;
\fill[line width=2.pt,color=ududff,fill=ududff,fill opacity=0.10000000149011612] (0.015,-0.625) -- (0.015,-0.5) -- (0.03,-0.5) -- (0.03,-0.625) -- cycle;
\fill[line width=2.pt,color=ududff,fill=ududff,fill opacity=0.10000000149011612] (0.015,-0.625) -- (0.015,-0.75) -- (0.03,-0.75) -- (0.03,-0.625) -- cycle;
\fill[line width=2.pt,color=ududff,fill=ududff,fill opacity=0.10000000149011612] (0.015,-0.875) -- (0.015,-0.75) -- (0.03,-0.75) -- (0.03,-0.875) -- cycle;
\fill[line width=2.pt,color=ududff,fill=ududff,fill opacity=0.10000000149011612] (0.015,-1.) -- (0.015,-0.875) -- (0.03,-0.875) -- (0.03,-1.) -- cycle;
\draw [line width=1.2pt,dash pattern=on 1pt off 1pt] (1.,-1.15) -- (1.,1.15);
\draw [line width=1.2pt,dash pattern=on 1pt off 1pt] (0.5,-1.15) -- (0.5,1.15);
\draw [line width=1.2pt,dash pattern=on 1pt off 1pt] (0.25,-1.15) -- (0.25,1.15);
\draw [line width=1.2pt,dash pattern=on 1pt off 1pt] (0.13,-1.15) -- (0.13,1.15);
\draw [line width=1.2pt,dash pattern=on 1pt off 1pt] (0.06,-1.15) -- (0.06,1.15);
\draw [line width=1.2pt,dash pattern=on 1pt off 1pt] (0.03,-1.15) -- (0.03,1.15);
\draw [line width=1.2pt,dash pattern=on 1pt off 1pt] (0.015,-1.15) -- (0.015,1.15);
\draw [line width=2.pt,color=ududff] (0.25,0.)-- (0.25,0.5);
\draw [line width=2.pt,color=ududff] (0.25,0.5)-- (0.5,0.5);
\draw [line width=2.pt,color=ududff] (0.5,0.5)-- (0.5,0.);
\draw [line width=2.pt,color=ududff] (0.5,0.)-- (0.25,0.);
\draw [line width=2.pt,color=ududff] (0.25,0.5)-- (0.25,1.);
\draw [line width=2.pt,color=ududff] (0.25,1.)-- (0.5,1.);
\draw [line width=2.pt,color=ududff] (0.5,1.)-- (0.5,0.5);
\draw [line width=2.pt,color=ududff] (0.5,0.5)-- (0.25,0.5);
\draw [line width=2.pt,color=ududff] (0.25,0.)-- (0.5,0.);
\draw [line width=2.pt,color=ududff] (0.5,0.)-- (0.5,-0.5);
\draw [line width=2.pt,color=ududff] (0.5,-0.5)-- (0.25,-0.5);
\draw [line width=2.pt,color=ududff] (0.25,-0.5)-- (0.25,0.);
\draw [line width=2.pt,color=ududff] (0.25,-1.)-- (0.25,-0.5);
\draw [line width=2.pt,color=ududff] (0.25,-0.5)-- (0.5,-0.5);
\draw [line width=2.pt,color=ududff] (0.5,-0.5)-- (0.5,-1.);
\draw [line width=2.pt,color=ududff] (0.5,-1.)-- (0.25,-1.);
\draw [line width=2.pt,color=ududff] (0.06,-1.)-- (0.06,-0.75);
\draw [line width=2.pt,color=ududff] (0.06,-0.75)-- (0.13,-0.75);
\draw [line width=2.pt,color=ududff] (0.13,-0.75)-- (0.13,-1.);
\draw [line width=2.pt,color=ududff] (0.13,-1.)-- (0.06,-1.);
\draw [line width=2.pt,color=ududff] (0.06,-0.75)-- (0.06,-0.5);
\draw [line width=2.pt,color=ududff] (0.06,-0.5)-- (0.13,-0.5);
\draw [line width=2.pt,color=ududff] (0.13,-0.5)-- (0.13,-0.75);
\draw [line width=2.pt,color=ududff] (0.13,-0.75)-- (0.06,-0.75);
\draw [line width=2.pt,color=ududff] (0.06,-0.5)-- (0.06,-0.25);
\draw [line width=2.pt,color=ududff] (0.06,-0.25)-- (0.13,-0.25);
\draw [line width=2.pt,color=ududff] (0.13,-0.25)-- (0.13,-0.5);
\draw [line width=2.pt,color=ududff] (0.13,-0.5)-- (0.06,-0.5);
\draw [line width=2.pt,color=ududff] (0.06,-0.25)-- (0.06,0.);
\draw [line width=2.pt,color=ududff] (0.06,0.)-- (0.13,0.);
\draw [line width=2.pt,color=ududff] (0.13,0.)-- (0.13,-0.25);
\draw [line width=2.pt,color=ududff] (0.13,-0.25)-- (0.06,-0.25);
\draw [line width=2.pt,color=ududff] (0.06,0.)-- (0.06,0.25);
\draw [line width=2.pt,color=ududff] (0.06,0.25)-- (0.13,0.25);
\draw [line width=2.pt,color=ududff] (0.13,0.25)-- (0.13,0.);
\draw [line width=2.pt,color=ududff] (0.13,0.)-- (0.06,0.);
\draw [line width=2.pt,color=ududff] (0.06,0.25)-- (0.06,0.5);
\draw [line width=2.pt,color=ududff] (0.06,0.5)-- (0.13,0.5);
\draw [line width=2.pt,color=ududff] (0.13,0.5)-- (0.13,0.25);
\draw [line width=2.pt,color=ududff] (0.13,0.25)-- (0.06,0.25);
\draw [line width=2.pt,color=ududff] (0.06,0.5)-- (0.06,0.75);
\draw [line width=2.pt,color=ududff] (0.06,0.75)-- (0.13,0.75);
\draw [line width=2.pt,color=ududff] (0.13,0.75)-- (0.13,0.5);
\draw [line width=2.pt,color=ududff] (0.13,0.5)-- (0.06,0.5);
\draw [line width=2.pt,color=ududff] (0.06,0.75)-- (0.06,1.);
\draw [line width=2.pt,color=ududff] (0.06,1.)-- (0.13,1.);
\draw [line width=2.pt,color=ududff] (0.13,1.)-- (0.13,0.75);
\draw [line width=2.pt,color=ududff] (0.13,0.75)-- (0.06,0.75);
\draw [line width=2.pt,color=ududff] (0.015,1.)-- (0.03,1.);
\draw [line width=2.pt,color=ududff] (0.03,1.)-- (0.03,0.875);
\draw [line width=2.pt,color=ududff] (0.03,0.875)-- (0.015,0.875);
\draw [line width=2.pt,color=ududff] (0.015,0.875)-- (0.015,1.);
\draw [line width=2.pt,color=ududff] (0.015,0.875)-- (0.015,0.75);
\draw [line width=2.pt,color=ududff] (0.015,0.75)-- (0.03,0.75);
\draw [line width=2.pt,color=ududff] (0.03,0.75)-- (0.03,0.875);
\draw [line width=2.pt,color=ududff] (0.03,0.875)-- (0.015,0.875);
\draw [line width=2.pt,color=ududff] (0.015,0.75)-- (0.015,0.625);
\draw [line width=2.pt,color=ududff] (0.015,0.625)-- (0.03,0.625);
\draw [line width=2.pt,color=ududff] (0.03,0.625)-- (0.03,0.75);
\draw [line width=2.pt,color=ududff] (0.03,0.75)-- (0.015,0.75);
\draw [line width=2.pt,color=ududff] (0.015,0.625)-- (0.015,0.5);
\draw [line width=2.pt,color=ududff] (0.015,0.5)-- (0.03,0.5);
\draw [line width=2.pt,color=ududff] (0.03,0.5)-- (0.03,0.625);
\draw [line width=2.pt,color=ududff] (0.03,0.625)-- (0.015,0.625);
\draw [line width=2.pt,color=ududff] (0.015,0.5)-- (0.015,0.375);
\draw [line width=2.pt,color=ududff] (0.015,0.375)-- (0.03,0.375);
\draw [line width=2.pt,color=ududff] (0.03,0.375)-- (0.03,0.5);
\draw [line width=2.pt,color=ududff] (0.03,0.5)-- (0.015,0.5);
\draw [line width=2.pt,color=ududff] (0.015,0.375)-- (0.015,0.25);
\draw [line width=2.pt,color=ududff] (0.015,0.25)-- (0.03,0.25);
\draw [line width=2.pt,color=ududff] (0.03,0.25)-- (0.03,0.375);
\draw [line width=2.pt,color=ududff] (0.03,0.375)-- (0.015,0.375);
\draw [line width=2.pt,color=ududff] (0.015,0.25)-- (0.015,0.125);
\draw [line width=2.pt,color=ududff] (0.015,0.125)-- (0.03,0.125);
\draw [line width=2.pt,color=ududff] (0.03,0.125)-- (0.03,0.25);
\draw [line width=2.pt,color=ududff] (0.03,0.25)-- (0.015,0.25);
\draw [line width=2.pt,color=ududff] (0.015,0.125)-- (0.015,0.);
\draw [line width=2.pt,color=ududff] (0.015,0.)-- (0.03,0.);
\draw [line width=2.pt,color=ududff] (0.03,0.)-- (0.03,0.125);
\draw [line width=2.pt,color=ududff] (0.03,0.125)-- (0.015,0.125);
\draw [line width=2.pt,color=ududff] (0.015,-0.125)-- (0.015,0.);
\draw [line width=2.pt,color=ududff] (0.015,0.)-- (0.03,0.);
\draw [line width=2.pt,color=ududff] (0.03,0.)-- (0.03,-0.125);
\draw [line width=2.pt,color=ududff] (0.03,-0.125)-- (0.015,-0.125);
\draw [line width=2.pt,color=ududff] (0.015,-0.25)-- (0.015,-0.125);
\draw [line width=2.pt,color=ududff] (0.015,-0.125)-- (0.03,-0.125);
\draw [line width=2.pt,color=ududff] (0.03,-0.125)-- (0.03,-0.25);
\draw [line width=2.pt,color=ududff] (0.03,-0.25)-- (0.015,-0.25);
\draw [line width=2.pt,color=ududff] (0.015,-0.25)-- (0.015,-0.375);
\draw [line width=2.pt,color=ududff] (0.015,-0.375)-- (0.03,-0.375);
\draw [line width=2.pt,color=ududff] (0.03,-0.375)-- (0.03,-0.25);
\draw [line width=2.pt,color=ududff] (0.03,-0.25)-- (0.015,-0.25);
\draw [line width=2.pt,color=ududff] (0.015,-0.375)-- (0.03,-0.375);
\draw [line width=2.pt,color=ududff] (0.03,-0.375)-- (0.03,-0.5);
\draw [line width=2.pt,color=ududff] (0.03,-0.5)-- (0.015,-0.5);
\draw [line width=2.pt,color=ududff] (0.015,-0.5)-- (0.015,-0.375);
\draw [line width=2.pt,color=ududff] (0.015,-0.625)-- (0.015,-0.5);
\draw [line width=2.pt,color=ududff] (0.015,-0.5)-- (0.03,-0.5);
\draw [line width=2.pt,color=ududff] (0.03,-0.5)-- (0.03,-0.625);
\draw [line width=2.pt,color=ududff] (0.03,-0.625)-- (0.015,-0.625);
\draw [line width=2.pt,color=ududff] (0.015,-0.625)-- (0.015,-0.75);
\draw [line width=2.pt,color=ududff] (0.015,-0.75)-- (0.03,-0.75);
\draw [line width=2.pt,color=ududff] (0.03,-0.75)-- (0.03,-0.625);
\draw [line width=2.pt,color=ududff] (0.03,-0.625)-- (0.015,-0.625);
\draw [line width=2.pt,color=ududff] (0.015,-0.875)-- (0.015,-0.75);
\draw [line width=2.pt,color=ududff] (0.015,-0.75)-- (0.03,-0.75);
\draw [line width=2.pt,color=ududff] (0.03,-0.75)-- (0.03,-0.875);
\draw [line width=2.pt,color=ududff] (0.03,-0.875)-- (0.015,-0.875);
\draw [line width=2.pt,color=ududff] (0.015,-1.)-- (0.015,-0.875);
\draw [line width=2.pt,color=ududff] (0.015,-0.875)-- (0.03,-0.875);
\draw [line width=2.pt,color=ududff] (0.03,-0.875)-- (0.03,-1.);
\draw [line width=2.pt,color=ududff] (0.03,-1.)-- (0.015,-1.);
\begin{scriptsize}
\draw [fill=ududff] (0.,0.) circle (1.0pt);
\draw[color=ududff] (-0.02,0.06) node {$O$};
\draw [fill=ududff] (1.,0.) circle (1.0pt);
\draw [fill=ududff] (0.5,0.) circle (1.0pt);
\draw [fill=ududff] (0.25,0.) circle (1.0pt);
\draw [fill=ududff] (0.13,0.) circle (1.0pt);
\draw [fill=ududff] (0.06,0.) circle (1.0pt);
\draw [fill=ududff] (0.03,0.) circle (1.0pt);
\draw [fill=ududff] (0.015,0.) circle (1.0pt);
\draw [fill=ududff] (0.5,0.5) circle (1.0pt);
\draw [fill=ududff] (0.5,1.) circle (1.0pt);
\draw [fill=ududff] (0.5,1.5) circle (1.0pt);
\draw [fill=ududff] (0.5,2.) circle (1.0pt);
\draw [fill=ududff] (0.5,-0.5) circle (1.0pt);
\draw [fill=ududff] (0.5,-1.) circle (1.0pt);
\draw [fill=ududff] (0.5,-1.5) circle (1.0pt);
\draw [fill=ududff] (0.5,-2.) circle (1.0pt);
\draw [fill=ududff] (0.25,0.5) circle (1.0pt);
\draw [fill=ududff] (0.25,1.) circle (1.0pt);
\draw [fill=ududff] (0.25,1.5) circle (1.0pt);
\draw [fill=ududff] (0.25,2.) circle (1.0pt);
\draw [fill=ududff] (0.25,-0.5) circle (1.0pt);
\draw [fill=ududff] (0.25,-1.) circle (1.0pt);
\draw [fill=ududff] (0.25,-1.5) circle (1.0pt);
\draw [fill=ududff] (0.13,0.5) circle (1.0pt);
\draw [fill=ududff] (0.13,1.) circle (1.0pt);
\draw [fill=ududff] (0.13,1.5) circle (1.0pt);
\draw [fill=ududff] (0.13,2.) circle (1.0pt);
\draw [fill=ududff] (0.13,-0.5) circle (1.0pt);
\draw [fill=ududff] (0.13,-1.) circle (1.0pt);
\draw [fill=ududff] (0.13,-1.5) circle (1.0pt);
\draw [fill=ududff] (0.13,-2.) circle (1.0pt);
\draw [fill=ududff] (0.06,0.5) circle (1.0pt);
\draw [fill=ududff] (0.06,1.) circle (1.0pt);
\draw [fill=ududff] (0.06,1.5) circle (1.0pt);
\draw [fill=ududff] (0.06,2.) circle (1.0pt);
\draw [fill=ududff] (0.06,-0.5) circle (1.0pt);
\draw [fill=ududff] (0.06,-1.) circle (1.0pt);
\draw [fill=ududff] (0.06,-1.5) circle (1.0pt);
\draw [fill=ududff] (0.06,-2.) circle (1.0pt);
\draw [fill=ududff] (0.25,-2.) circle (1.0pt);
\draw [fill=ududff] (0.13,1.75) circle (1.0pt);
\draw [fill=ududff] (0.13,1.25) circle (1.0pt);
\draw [fill=ududff] (0.13,0.75) circle (1.0pt);
\draw [fill=ududff] (0.13,0.25) circle (1.0pt);
\draw [fill=ududff] (0.13,-0.25) circle (1.0pt);
\draw [fill=ududff] (0.13,-0.75) circle (1.0pt);
\draw [fill=ududff] (0.13,-1.25) circle (1.0pt);
\draw [fill=ududff] (0.13,-1.75) circle (1.0pt);
\draw [fill=ududff] (0.06,-1.75) circle (1.0pt);
\draw [fill=ududff] (0.06,-1.25) circle (1.0pt);
\draw [fill=ududff] (0.06,-0.75) circle (1.0pt);
\draw [fill=ududff] (0.06,-0.25) circle (1.0pt);
\draw [fill=ududff] (0.06,0.25) circle (1.0pt);
\draw [fill=ududff] (0.06,0.75) circle (1.0pt);
\draw [fill=ududff] (0.06,1.25) circle (1.0pt);
\draw [fill=ududff] (0.06,1.75) circle (1.0pt);
\draw [fill=ududff] (0.03,0.5) circle (1.0pt);
\draw [fill=ududff] (0.03,1.) circle (1.0pt);
\draw [fill=ududff] (0.03,-1.) circle (1.0pt);
\draw [fill=ududff] (0.03,-0.5) circle (1.0pt);
\draw [fill=ududff] (0.015,1.) circle (1.0pt);
\draw [fill=ududff] (0.015,0.5) circle (1.0pt);
\draw [fill=ududff] (0.015,-1.) circle (1.0pt);
\draw [fill=ududff] (0.015,-0.5) circle (1.0pt);
\draw [fill=ududff] (0.015,-0.75) circle (1.0pt);
\draw [fill=ududff] (0.015,-0.25) circle (1.0pt);
\draw [fill=ududff] (0.015,0.25) circle (1.0pt);
\draw [fill=ududff] (0.015,0.75) circle (1.0pt);
\draw [fill=ududff] (0.03,0.75) circle (1.0pt);
\draw [fill=ududff] (0.03,0.25) circle (1.0pt);
\draw [fill=ududff] (0.03,-0.25) circle (1.0pt);
\draw [fill=ududff] (0.03,-0.75) circle (1.0pt);
\draw [fill=ududff] (0.03,0.875) circle (1.0pt);
\draw [fill=ududff] (0.03,0.625) circle (1.0pt);
\draw [fill=ududff] (0.03,0.375) circle (1.0pt);
\draw [fill=ududff] (0.03,0.125) circle (1.0pt);
\draw [fill=ududff] (0.03,-0.125) circle (1.0pt);
\draw [fill=ududff] (0.03,-0.375) circle (1.0pt);
\draw [fill=ududff] (0.03,-0.625) circle (1.0pt);
\draw [fill=ududff] (0.03,-0.875) circle (1.0pt);
\draw [fill=ududff] (0.015,-0.875) circle (1.0pt);
\draw [fill=ududff] (0.015,-0.625) circle (1.0pt);
\draw [fill=ududff] (0.015,-0.375) circle (1.0pt);
\draw [fill=ududff] (0.015,-0.125) circle (1.0pt);
\draw [fill=ududff] (0.015,0.125) circle (1.0pt);
\draw [fill=ududff] (0.015,0.375) circle (1.0pt);
\draw [fill=ududff] (0.015,0.625) circle (1.0pt);
\draw [fill=ududff] (0.015,0.875) circle (1.0pt);
\draw[color=ududff] (0.75,0.25) node {$L_0$};
\draw[color=ududff] (0.195,0.25) node {$L_1$};
\draw[color=ududff] (0.375,0.25) node {$I_{0,1}$};
\draw[color=ududff] (0.375,0.75) node {$I_{0,2}$};
\draw[color=ududff] (0.375,-0.25) node {$I_{0,0}$};

\draw[color=ududff] (0.1,-0.15642539785446855) node {$I_{1,0}$};
\draw[color=ududff] (0.1,0.08937157613029365) node {$I_{1,1}$};
\draw[color=ududff] (0.1,0.33516855011505586) node {$I_{1,2}$};
\draw[color=ududff] (0.1,0.5998729836371074) node {$I_{1,3}$};
\draw[color=ududff] (0.1,0.85) node {$I_{1,4}$};
\end{scriptsize}
\end{axis}
\end{tikzpicture}
\caption{}\label{F:2}

\end{figure}

Let  $I=[0,1]^2$, $L=[0,1]\times\mathbb{R}$   as in Lemmas \ref{L:main1} and  \ref{L:main2}, and for any $n\in\mathbb{N}$ and $k\in\mathbb{Z}$ let
\[r_{n}\colon L_n \to L \quad \text{ and } \quad r_{n,k}\colon I_{n,k}\to I,\]
be the  homeomorphisms which are given by the formulas:
\begin{eqnarray}
r_0(x,y)&=&\big(2x-1,y)\\
r_n(x,y)&=&\big(\frac{4^{n+1}}{2}\cdot x-1,
(n+2-\frac{4^{n+1}}{2}x)y\big), \text{ if } n>0\label{EQ:Distortion}\\
r_{n,k}(x,y)&=&\big(4^{n+1}\cdot(x-\frac{1}{4^{n+1}}),
(n+1)(y-\frac{k-1}{n+1})    \big).\label{EQ:Translation}
\end{eqnarray}
In other words, $r_n$  maps the $x$-interval $[2/4^{n+1},1/4^{n}]$ affinely onto $[0,1]$
 and, for $n>0$, it interpolates between matching the $y$-coordinates $\frac{1}{n+1}\cdot\mathbb{Z}$ of the left boundary of $L_n$ with $\mathbb{Z}$ and the $y$-coordinates $\frac{1}{n}\cdot\mathbb{Z}$ of the right boundary of $L_n$ with $\mathbb{Z}$. Similarly $r_{n,k}$  translates  $I_{n,k}$ vertically to $I_{n,1}$ and then post-composes with  the obvious homeomorphism  $I_{n,1}\xrightarrow{\raisebox{-0.5ex}[0ex][0ex]{$\sim$}} I$.

For each  $n\in\mathbb{N}$, let $\mathcal{B}_n:= r_n^*\mathcal{B}_L$ of $L_n$ be the deballing recipe of $L_n$ that is the pullback of $\mathcal{B}_L$, from Lemma \ref{L:main2}, under $r_n$. Similarly, for every $n\in\mathbb{N}$ and $k\in\mathbb{Z}$, let  $\mathcal{B}_{n,k}:= r_{n,k}^*\mathcal{B}_I$ be the deballing recipe of $I_{n,k}$ which is the pullback of $\mathcal{B}_I$ from Lemma \ref{L:main1} under $r_{n,k}$. Set
\[\mathcal{B}_K:=\big(\bigcup_n\mathcal{B}_n\big)\cup\big(\bigcup_{n,k}\mathcal{B}_{n,k}\big)\]
be the associated deballing recipe on $K$.

Finally, we let $J$ be the subset of the plane that is the union of the right half-disc $D:=\{(x,y)\in \mathbb{R}^2\colon |x|^2+|y|^2\leq 1 \text{ and } x\geq 0\}$ and the squares $I_{+}:=[-1,0]\times [0,1]$ and $I_{-}:=[-1,0]\times [-1,0]$, as in Figure \ref{F:D}. In what follows, we will identify $D$ with the two-point compactification $\widehat{K}$ of $K$ under the   homeomorphism $\kappa\colon D\to \widehat{K}$ from (\ref{EQ:compactification}). Similarly we will identify $I_{+}$ and $I_{-}$ with the  $I:=[0,1]^2$ via the translation maps $t_+\colon I_{+}\to I$ and $t_-\colon I_{-}\to I$ given by $t_{+}(x,y)=(x+1,y)$ and $t_{-}(x,y)=(x+1,y+1)$. Consider the deballing recipe 
\[\mathcal{B}_J:= \mathcal{B}_{+}\cup \mathcal{B}_{-}\cup\mathcal{B}_{K}\]
on $J$ where $\mathcal{B}_{+}, \mathcal{B}_{-}$ are pullbacks of $\mathcal{B}_I$ from Lemma \ref{L:main1} under $t_{+}$ and $t_{-}$, and set 
\begin{equation}
\mathcal{S}:=\mathcal{B}_J(J).
\end{equation}
By Theorem \ref{T:whyburn}, $\mathcal{S}$ is homeomorphic to the Sierpinski carpet. Having fixed the identification  $\kappa\colon D\to \widehat{K}$, below we will not be distinguishing between $D\setminus\{(0,-1),(0,+1)\}$ and $K$.

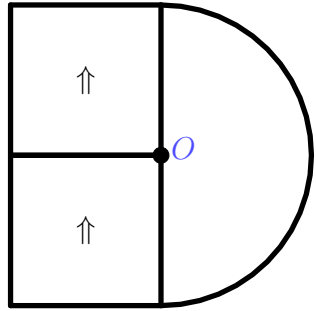
\begin{figure}[ht!]
\centering
\begin{tikzpicture}[line cap=round,line join=round,>=triangle 45,x=1cm,y=1cm]
\draw [shift={(0,0)},line width=2pt]  plot[domain=-1.5707963267948966:1.5707963267948966,variable=\t]({1*2*cos(\t r)+0*2*sin(\t r)},{0*2*cos(\t r)+1*2*sin(\t r)});
\draw [line width=2pt] (0,2)-- (-2,2);
\draw [line width=2pt] (-2,2)-- (-2,0);
\draw [line width=2pt] (-2,0)-- (0,0);
\draw [line width=2pt] (0,0)-- (0,2);
\draw (-1,1) node {$\Uparrow$};
\draw (-1,-1) node {$\Uparrow$};
\draw [line width=2pt] (-2,0)-- (-2,-2);
\draw [line width=2pt] (-2,-2)-- (0,-2);
\draw [line width=2pt] (0,-2)-- (0,0);
\draw [line width=2pt] (0,0)-- (-2,0);
\draw [fill=black] (0,0) circle (3pt);
\draw[color=ududff] (+0.3,0.1) node {$O$};
\end{tikzpicture}
\caption{A dynamical realization of $\mathcal{S}$}\label{F:D}
\end{figure}

Using this ``dynamical realization" of the Sierpinski carpet we may  define the desired Borel reduction 
$\rho\colon G\to \mathcal{H}(\mathcal{S})$ as follows: for any $g\in G$ let $\rho(g)\in\mathcal{H}(\mathcal{S})$ with

\[\rho(g)(z)=
\begin{cases}
\big(t_+^{-1} \circ \pi \circ t_+\big) (z) & \text{ if } z\in I_+  \\
\big(t_{-}^{-1} \circ \pi \circ t_{-} \big)(z) & \text{ if }  z\in I_{-}  \\
\tau(g)(z) & \text{ if }  z\in K  \\
\end{cases}
\]
where $\tau(g)$ is the homeomorphism of $K$  given by sending $z\in K$ to
\[\tau(g)(z)=
\begin{cases}
z  & \text{ if } z\in L_n \text{ for some } n\in\mathbb{N}  \\
\big(r_{n,k}^{-1} \circ \pi^{-1} \circ r_{n,k} \big)(z) & \text{ if }  z\in I_{n,k} \text{ for }   n\in\mathbb{N}, k\in\mathbb{Z} \text{ with } g(n)=k \\
\big(r_{n,k}^{-1} \circ \pi \circ r_{n,k} \big)(z) & \text{ if }  z\in I_{n,k} \text{ for }   n\in\mathbb{N}, k\in\mathbb{Z} \text{ with } g(n)\neq k. 
\end{cases}
\]
Here $\pi\in\mathcal{H}(I)$ is as in Lemma \ref{L:main1}. In particular, Lemma \ref{L:main1}(1)  guarantees that both $\tau$ and $\rho$ are well-defined homeomorphisms of $K$ and $J$ respectively. It is straightforward that $\rho\colon G \to \mathcal{H}(\mathcal{S})$ is continuous. Next we show that $\rho$ is a reduction.

It is not difficult to see that if $\varphi,\varphi'\in\mathcal{H}(K)$ are any homeomorphisms in the image of $\rho$, so  that $\varphi' =\sigma \varphi \sigma^{-1}$ holds for some $\sigma\in \mathcal{H}(K)$, then $\sigma$ has to  fix $K$ setwise and  has to map the point $O=(0,0)$ to itself: $\{\sigma(z) \colon z\in K\}=K$ and $\sigma(O)=O$.  This follows from the unique behavior of homeomorphisms of the form $\rho(g)$ on $I_{+},I_{-}$, as well as on a neighborhood of any point of the semicircle $\{(x,y)\in \partial D \colon x>0\}$.

 In particular,  in order  to establish Proposition \ref{Prop:reduction}, it suffices to prove that for all $g,g'\in G$ we have that $gE_{\widehat{G}}g'$ holds if and only if there exists  $\sigma\in\mathcal{H}(K)$ with $\sigma(O)=O$ so that $ \tau(g)= \sigma \circ \tau(g') \circ \sigma^{-1}$. This is precisely the content of the next two claims.

\begin{claim}
If $g,g'\in G$ with $gE_{\widehat{G}}g'$, then there  exists  $\sigma\in\mathcal{H}(K)$ with $\sigma(O)=O$ and \[\tau(g)= \sigma \circ \tau (g') \circ \sigma^{-1}.\]
\end{claim}
\begin{proof}
Let $g,g'\in G$ with $gE_{\widehat{G}}g'$. Then $g-g' \in \widehat{G}$. Let $s\colon \mathbb{N}\cup\{-1\}\to\mathbb{Z}$ by setting  $s(-1)=0$ and  $s(n)=g(n)-g'(n)$, if $n\in\mathbb{N}$.
Define $\sigma\colon K\to K$ by setting
\[\sigma(x,y)=
\begin{cases}
(x,y)+(0,\frac{s(n)}{n+1})  &(x,y)\in I_{n,k} \text{ for some }   n\in\mathbb{N}, k\in\mathbb{Z}, \\
(x,y) & \text{ if } x=0, \\
\big(r^{-1}_n\circ   h^{\mathcal{B}}_{s(n),s(n-1)} \circ r_n\big)
(x,y) & \text{ if }  \text{ if } (x,y)\in L_n \text{ for some } n\in\mathbb{N}, 
\end{cases}
\]
where $r_n$ is as in (\ref{EQ:Distortion})  and $ h^{\mathcal{B}}_{k,\ell}$ as in Lemma \ref{L:main2}. To see that $\sigma$ is a well-defined function, one just checks that  for all $y\in\mathbb{R}$ and all $x\in\{1/(2\cdot 4^n) \colon n\in\mathbb{N}\}$ we have that
\[\big(r^{-1}_n\circ   h^{\mathcal{B}}_{s(n),s(n-1)} \circ r_n\big)
(x,y)=(x,y)+(0,\frac{s(n)}{n+1});\]
and that for all $y\in\mathbb{R}$ and all $x\in\{1/4^n \colon n\in\mathbb{N}\}$ we have that
\[\big(r^{-1}_n\circ   h^{\mathcal{B}}_{s(n),s(n-1)} \circ r_n\big)
(x,y)=(x,y)+(0,\frac{s(n-1)}{n}).\]
It is  now clear that $\sigma(O)=O$ holds and it is also   easy to check that for all $(x,y)\in K$ we have that  $\tau(g)(x,y)= \big(\sigma \circ \tau(g') \circ \sigma^{-1}\big)(x,y)$. Therefore, we are just left with showing that  $\sigma$ is a homeomorphism of $K$.
Notice that $\sigma$ restricts to a self-homeomorphism of $L_n$ and to  a self-homeomorphism of  $\bigcup_kI_{n,k}$, for all $n\in\mathbb{N}$. Hence, it suffices to show that that $\sigma$ is continuous on points of the form $(0,y)$ with $y\in \mathbb{R}$. 
To see this, set $(\tilde{x},\tilde{y}):=\sigma(x,y)$ and fix some $\varepsilon>0$ and  $M>0$. We will find  $n_0\in\mathbb{N}$ so that:
\begin{enumerate}
\item[(i)] If $(x,y)\in I_{n,k}$ for some $n>n_0$ and $k\in\mathbb{Z}$, then $|x-\tilde{x}|<\varepsilon$ and $|y-\tilde{y}|<\varepsilon$;
\item[(ii)] If $(x,y)\in L_n$  for some $n>n_0$ and $|y|<M$,  then $|x-\tilde{x}|<\varepsilon$ and $|y-\tilde{y}|<\varepsilon$.
\end{enumerate}

For (i), notice that if $(x,y)\in I_{n,k}$, then  $|x-\tilde{x}|\leq 1/4^{n+1}$ and   $|y-\tilde{y}|\leq |s(n)/(n+1)|$. Since $\sigma(n)=g(n)-g'(n)$ and $g-g'\in \widehat{G}$, we get 
some $n_0\in\mathbb{N}$ which  satisfies (i).
\smallskip{}

For (ii), let $(x,y)\in L_n$ and notice that $|x-\tilde{x}|\leq 1/(2\cdot 4^{n+1})$. So it suffices to show that $|y-\tilde{y}|\to 0$ also holds, when $n\to \infty$.
It will be  convenient to assume $y\geq 0$. A  symmetric argument similarly establishes the case $y\leq 0$. 

Since $0\leq y < M$ holds, we may pick some $k\in \mathbb{N}$ with $1\leq k<M/(n+1)$ so that $(x,y)$ lies in the trapezoid $T$ whose endpoints are
\begin{eqnarray*}
(\frac{1}{4^{n+1}}, \frac{k}{n+1}) &\quad &  (\frac{1}{4^{n}}, \frac{k}{n})\\
(\frac{1}{4^{n+1}}, \frac{k-1}{n+1}) &\quad & (\frac{1}{4^{n}}, \frac{k-1}{n})
\end{eqnarray*}
By Lemma \ref{L:main2} we have that $\sigma(T)$ is  contained in the  trapezoid $\widetilde{T}$ whose endpoints are:
\begin{eqnarray*}
(\frac{1}{4^{n+1}}, \frac{\max{\{s(n-1),s(n)\}}+k+1}{n+1}) &\quad & 
(\frac{1}{4^{n}}, \frac{\max{\{s(n-1),s(n)}+k+1\}}{n})\\
(\frac{1}{4^{n+1}}, \frac{\min{\{s(n-1),s(n)\}}+k-1}{n+1}) &\quad & (\frac{1}{4^{n}}, \frac{\min{\{s(n-1),s(n)\}}+k-1}{n})
\end{eqnarray*}

\begin{figure}[ht!]
\definecolor{ttttff}{rgb}{0.2,0.2,1}
\definecolor{zzttqq}{rgb}{0.6,0.2,0}
\centering
\begin{tikzpicture}[line cap=round,line join=round,>=triangle 45,x=1cm,y=1cm,yscale=0.5,xscale=1.2]
\fill[line width=2pt,color=zzttqq,fill=zzttqq,fill opacity=0.10000000149011612] (-3,1) -- (3,2) -- (3,4) -- (-3,2) -- cycle;
\fill[line width=2pt,color=ttttff,fill=ttttff,fill opacity=0.1] (-3,3) -- (3,6) -- (3,14) -- (-3,7) -- cycle;
\fill[line width=2pt,color=zzttqq,fill=zzttqq,fill opacity=0.10000000149011612] (-3,5) -- (-1.6945086685443915,6.818213571894482) -- (-0.696068320041519,8.303931679958481) -- (-0.3804248473169871,9.626392696709209) -- (0.6456406021345479,11.030482259116571) -- (3,12) -- (3,10) -- (2.6617692045656343,8.7623375813816) -- (2.355749684553773,7.790275576638042) -- (2.1937393504298464,6.800212423658491) -- (1.419689976282197,6.2421768283427435) -- (0.5736360091905806,5.234112527127201) -- (-0.43442829202496264,5.378121713015136) -- (-2.1625385226801797,5.0541010447672825) -- (-3,4) -- cycle;
\draw [line width=1pt] (-3,0)-- (3,0);
\draw [line width=1pt] (-3,1)-- (3,2);
\draw [line width=1pt] (-3,2)-- (3,4);
\draw [line width=1pt] (-3,3)-- (3,6);
\draw [line width=1pt] (-3,4)-- (3,8);
\draw [line width=1pt] (-3,5)-- (3,10);
\draw [line width=1pt] (-3,6)-- (3,12);
\draw [line width=1pt] (-3,7)-- (3,14);
\draw [line width=1pt] (-3,0)-- (-3,14);
\draw [line width=1pt] (3,0)-- (3,14);
\draw [line width=2pt,color=zzttqq] (-3,1)-- (3,2);
\draw [line width=2pt,color=zzttqq] (3,2)-- (3,4);
\draw [line width=2pt,color=zzttqq] (3,4)-- (-3,2);
\draw [line width=2pt,color=zzttqq] (-3,2)-- (-3,1);
\draw [line width=2pt,color=ttttff] (-3,3)-- (3,6);
\draw [line width=2pt,color=ttttff] (3,6)-- (3,14);
\draw [line width=2pt,color=ttttff] (3,14)-- (-3,7);
\draw [line width=2pt,color=ttttff] (-3,7)-- (-3,3);
\draw [line width=2pt,color=zzttqq] (-3,5)-- (-1.6945086685443915,6.818213571894482);
\draw [line width=2pt,color=zzttqq] (-1.6945086685443915,6.818213571894482)-- (-0.696068320041519,8.303931679958481);
\draw [line width=2pt,color=zzttqq] (-0.696068320041519,8.303931679958481)-- (-0.3804248473169871,9.626392696709209);
\draw [line width=2pt,color=zzttqq] (-0.3804248473169871,9.626392696709209)-- (0.6456406021345479,11.030482259116571);
\draw [line width=2pt,color=zzttqq] (0.6456406021345479,11.030482259116571)-- (3,12);
\draw [line width=2pt,color=zzttqq] (3,12)-- (3,10);
\draw [line width=2pt,color=zzttqq] (3,10)-- (2.6617692045656343,8.7623375813816);
\draw [line width=2pt,color=zzttqq] (2.6617692045656343,8.7623375813816)-- (2.355749684553773,7.790275576638042);
\draw [line width=2pt,color=zzttqq] (2.355749684553773,7.790275576638042)-- (2.1937393504298464,6.800212423658491);
\draw [line width=2pt,color=zzttqq] (2.1937393504298464,6.800212423658491)-- (1.419689976282197,6.2421768283427435);
\draw [line width=2pt,color=zzttqq] (1.419689976282197,6.2421768283427435)-- (0.5736360091905806,5.234112527127201);
\draw [line width=2pt,color=zzttqq] (0.5736360091905806,5.234112527127201)-- (-0.43442829202496264,5.378121713015136);
\draw [line width=2pt,color=zzttqq] (-0.43442829202496264,5.378121713015136)-- (-2.1625385226801797,5.0541010447672825);
\draw [line width=2pt,color=zzttqq] (-2.1625385226801797,5.0541010447672825)-- (-3,4);
\draw [line width=2pt,color=zzttqq] (-3,4)-- (-3,5);
\begin{scriptsize}
\draw [fill=black] (3,0) circle (1.5pt);
\draw [fill=black] (3,2) circle (1.5pt);
\draw [fill=black] (3,4) circle (1.5pt);
\draw [fill=black] (3,6) circle (1.5pt);
\draw [fill=black] (3,8) circle (1.5pt);
\draw [fill=black] (3,10) circle (1.5pt);
\draw [fill=black] (3,12) circle (1.5pt);
\draw [fill=black] (3,14) circle (1.5pt);
\draw [fill=black] (-3,0) circle (1.5pt);
\draw [fill=black] (-3,1) circle (1.5pt);
\draw [fill=black] (-3,2) circle (1.5pt);
\draw [fill=black] (-3,3) circle (1.5pt);
\draw [fill=black] (-3,4) circle (1.5pt);
\draw [fill=black] (-3,5) circle (1.5pt);
\draw [fill=black] (-3,6) circle (1.5pt);
\draw [fill=black] (-3,7) circle (1.5pt);
\end{scriptsize}
\end{tikzpicture}
\caption{The sets $T, \widetilde{T}\subseteq L_n$ and  $\sigma(T)\subseteq \widetilde{T}$, if $\sigma(n)=3$ and $\sigma(n-1)=4$.}
\end{figure}
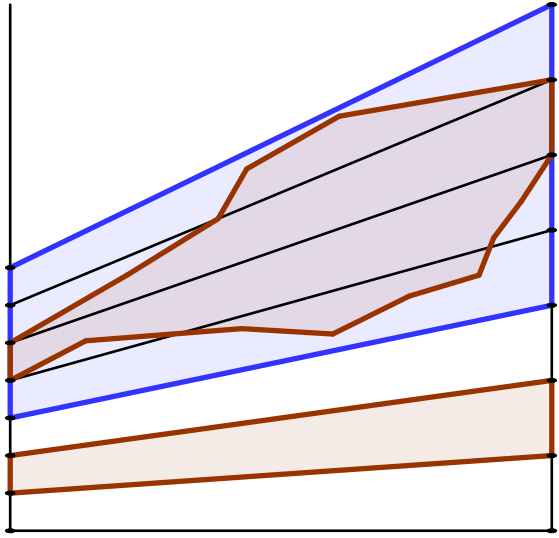

As a consequence, we have that $|y-\tilde{y}|\leq|p-q|$, where  $p,q$ correspond to the maximum and minimum value among the numbers in the following list,  where $s\in \{s(n-1), s(n)\}$:
\[\frac{k-1}{n+1},\frac{k}{n+1}, \frac{k-1}{n}, \frac{k}{n},  \frac{s+k+1}{n+1}, \frac{s+k-1}{n+1}, \frac{s+k+1}{n}, \frac{s+k-1}{n},  \] 
By triangle inequality, $|p-q|$ is bounded above by the sum of at most six terms among:
\[\frac{1}{n+1}, \frac{1}{n}, \frac{M}{(n+1)^2}, \frac{M}{n(n+1)}, \frac{|s(n)|}{n+1}, \frac{|s(n)|}{n}, \frac{|s(n-1)|}{n+1}, \frac{|s(n-1)|}{n}  \]
Since all these terms are independent of the choice of $y<M$, and  converge to $0$ as $n\to\infty$, we get $n_0\in\mathbb{N}$ so that $|y-\tilde{y}|<\varepsilon$ holds
for all  $(x,y)\in L_n$, with $n>n_0$ and  $|y|<M$.
\end{proof}  

\bigskip

\begin{claim}
If  $g,g'\in G$  with  $\tau(g)= \sigma \circ \tau(g')   \circ \sigma^{-1}$, for some  $\sigma\in\mathcal{H}(K)$ with $\sigma(O)=O$, then 
\[gE_{\widehat{G}}g'.\]
\end{claim}
\begin{proof} Suppose we are given $g,g' \in G$ and $\sigma\in\mathcal{H}(K)$ satisfying $\sigma(O)=O$ such that $\tau(g)= \sigma \circ \tau(g') \circ \sigma^{-1}$. By the structure of the homeomorphisms lying in $\tau(G)$ and a simple induction on $n\in\mathbb{N}$ we have that $\sigma$ restricts to a self-homeomorphism of each $L_n$ and to a self-homeomorphism of $\bigcup_k I_{n,k}$ for each $n \in \N$. Moreover,
since each $\mathcal{B}_{K}(I_{n,k})\setminus \partial(I_{n,k})$ is a maximal connected open subset of $K$ on which both $\tau(g)$ and $\tau(g')$ have no fixed points, for every $n\in \mathbb{N}$ and every $k\in\mathbb{Z}$ there is $\ell\in\mathbb{Z}$ so that 
\[\sigma(\mathcal{B}_{K}(I_{n,k})) = \mathcal{B}_{K}(I_{n,\ell}).\]
Notice  that we may  characterize $\mathcal{B}_{K}(I_{n,g'(n)})$ inside the family $\set{\mathcal{B}_{K}(I_{n,k}) \mid k \in \bZ}$  as the member indexed by the unique $k$ for which there exists  some $m\in\mathbb{Z}$ with  $m\neq k$, so that 
\[ \set{\lim_{i\to+\infty}\tau(g)^i(z) \mid z \in \mathcal{B}_{K}(I_{n,k})\setminus \del(I_{n,k})} = \set{\lim_{i\to+\infty}\tau(g)^i(z) \mid z \in \mathcal{B}_{K}(I_{n,m})\setminus \del(I_{n,m})}. \]
Similarly we may characterize $\mathcal{B}_{K}(I_{n,g(n)})$ inside the family $\set{\mathcal{B}_{K}(I_{n,\ell}) \mid \ell \in \bZ}$. Since these characterizations are topological, they have to be respected by the intertwiner $\sigma$.  That is 
\[\sigma(\mathcal{B}_{K}(I_{n,g'(n)})) = \mathcal{B}_{K}(I_{n,g(n)})\]
should hold for all $n\in\mathbb{N}$. It follows that for all $n\in\mathbb{N}$ and  $m\in\mathbb{Z}$ we have that:
\begin{equation}
\sigma(\mathcal{B}_{K}(I_{n,g'(n)+m})) = \mathcal{B}_{K}(I_{n,g(n)+m}).
\end{equation}
Setting  $m:=-g'(n)$ for each $n\in\mathbb{N}$ and recalling that  $\sigma(O)=O$, we have that
\[ \lim_{n\to \infty} d(\mathcal{B}_{K}(I_{n,0}),\hspace{0.5mm}\mathcal{B}_{K}(I_{n,g(n)-g'(n)})) = 0.\]
It follows that
\[ \frac{g(n)-g'(n)}{n+1} \to 0,\]
which amounts to asserting that $g-g' \in \widehat{G}$. That is, $gE_{\widehat{G}}g'$ holds.

\end{proof}

The proof of Theorem \ref{T:Main} is now complete, assuming Lemmas \ref{L:main1} and  \ref{L:main2}. We prove these lemmas in the subsequent Sections \ref{SL:main1} and  \ref{SL:main2}, respectively.

\section{Proof of Lemma \ref{L:main1}}\label{SL:main1}
It will be convenient to work with  $J:=[-1,1]^2$ instead of $[0,1]^2$. We will define
a deballing recipe $\mathcal{B}$ of $J$ and   $\pi\in\mathcal{H}\big(\mathcal{B}(J)\big)$ 
which fixes $\partial J$ pointwise and for  all $z=(x,y)\in\mathcal{B}(J) \setminus \partial J$,
\[\lim_{n\to-\infty}\pi^n(z)=(x,0)\; \text{ and  }\; \lim_{n\to\infty}\pi^n(z)=(x,1).\]

\begin{proof}[Proof of Lemma \ref{L:main1}]
For every $k\in\mathbb{N}$ we define a sequence  $(s^k_n)_{n\in\mathbb{N}}$ in $[0,1]$ as follows. When $k=0$, we let $s^0_0=0$ and  $s^0_n=1/2^n$ for all $n>0$. Assuming now that $(s^k_n)_{n\in\mathbb{N}}$ has been defined for some $k\in\mathbb{N}$, we let $s^{k+1}_0=0$ and $s^{k+1}_{2n-1}=s^{k+1}_{2n}=s^k_n/2$  for all $n>0$. For example,
\[s^1_0=0, \quad s^1_1=1/4, \quad s^1_2=1/4, \quad s^1_3=1/8, \quad s^1_4=1/8, \ldots.\] 
We then decompose $J$ into a family of trapezoids $T_{a,b}$ indexed by pairs of integers  $(a,b)\in \bZ^2$ as in Figure \ref{F:Trapezoids} below. Concretely, for every $(a,b)\in \mathbb{N}^2$ the vertices of $T_{a,b}$ are given by: \medskip 

\begin{center}
$\begin{array}{cc}
   \displaystyle v_{a,b}^{0,1} = (\sum_{j=0}^{a} s^0_j, \sum_{j=0}^{b+1} s^a_j)  &  \displaystyle v_{a,b}^{1,1} = (\sum_{j=0}^{a+1} s^0_j, \sum_{j=0}^{b+1} s^{a+1}_j) \\
   \displaystyle v_{a,b}^{0,0} = (\sum_{j=0}^a s^0_j, \sum_{j=0}^{b} s^a_j) & \displaystyle v_{a,b}^{1,0} = (\sum_{j=0}^{a+1} s^0_j, \sum_{j=0}^{b} s^{a+1}_j)
\end{array}$
\end{center}
The remaining trapezoids  are defined analogously. For example,  $T_{1,2}$ has vertices:
\begin{align*}
     v_{1,2}^{0,1} &= (\frac{1}{2}, \, \frac{1}{4}+\frac{1}{4}+\frac{1}{8})   \quad   & v_{1,2}^{1,1} & = (\frac{1}{2}+\frac{1}{4}, \, \frac{1}{8}+\frac{1}{8}+\frac{1}{8})\\
     v_{1,2}^{0,0} &= (\frac{1}{2}, \, \frac{1}{4}+\frac{1}{4}) \quad  &v_{1,2}^{1,0}& = (\frac{1}{2}+\frac{1}{4},  \, \frac{1}{8}+\frac{1}{8})
\end{align*}

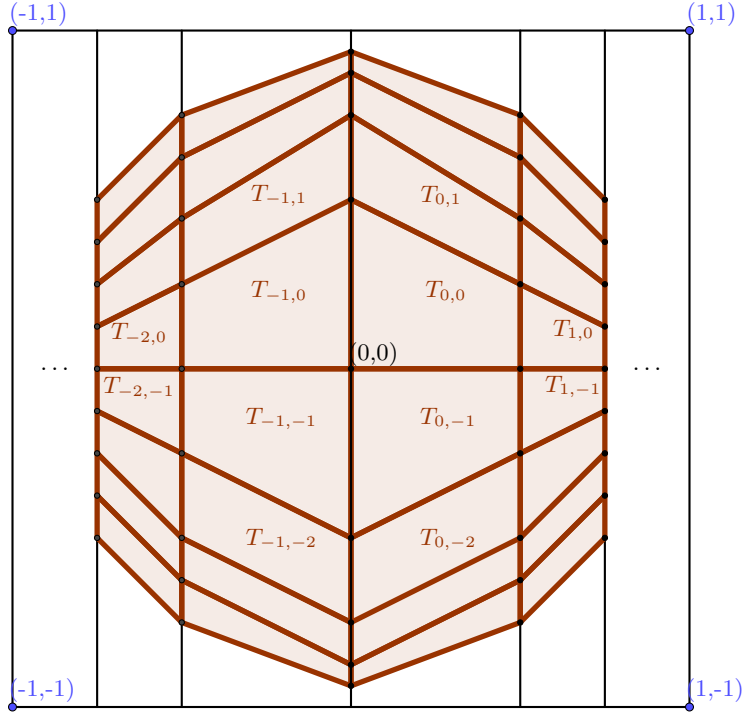
\begin{figure}[ht!]
\begin{tikzpicture}[line cap=round,line join=round,>=triangle 45,x=4.5cm,y=4.5cm]
\clip(-1.2,-1.2) rectangle (1.2,1.2);
\fill[line width=2.pt,color=zzttqq,fill=zzttqq,fill opacity=0.10000000149011612] (0.,0.) -- (0.5,0.) -- (0.5,0.25) -- (0.,0.5) -- cycle;
\fill[line width=2.pt,color=zzttqq,fill=zzttqq,fill opacity=0.10000000149011612] (0.,0.5) -- (0.5,0.25) -- (0.5,0.4444444444) -- (0.,0.75) -- cycle;
\fill[line width=2.pt,color=zzttqq,fill=zzttqq,fill opacity=0.10000000149011612] (0.,0.75) -- (0.5,0.4444444444) -- (0.5,0.625) -- (0.,0.875) -- cycle;
\fill[line width=2.pt,color=zzttqq,fill=zzttqq,fill opacity=0.10000000149011612] (0.,0.9375) -- (0.,0.875) -- (0.5,0.625) -- (0.5,0.75) -- cycle;
\fill[line width=2.pt,color=zzttqq,fill=zzttqq,fill opacity=0.10000000149011612] (0.,-0.5) -- (0.,-0.75) -- (0.5,-0.5) -- (0.5,-0.25) -- cycle;
\fill[line width=2.pt,color=zzttqq,fill=zzttqq,fill opacity=0.10000000149011612] (0.,-0.75) -- (0.,-0.875) -- (0.5,-0.625) -- (0.5,-0.5) -- cycle;
\fill[line width=2.pt,color=zzttqq,fill=zzttqq,fill opacity=0.10000000149011612] (0.,-0.875) -- (0.,-0.9375) -- (0.5,-0.75) -- (0.5,-0.625) -- cycle;
\fill[line width=2.pt,color=zzttqq,fill=zzttqq,fill opacity=0.10000000149011612] (0.,0.) -- (0.5,0.) -- (0.5,-0.25) -- (0.,-0.5) -- cycle;
\fill[line width=2.pt,color=zzttqq,fill=zzttqq,fill opacity=0.10000000149011612] (0.5,0.) -- (0.75,0.) -- (0.75,0.125) -- (0.5,0.25) -- cycle;
\fill[line width=2.pt,color=zzttqq,fill=zzttqq,fill opacity=0.10000000149011612] (0.5,0.25) -- (0.75,0.125) -- (0.75,0.25) -- (0.5,0.4444444444) -- cycle;
\fill[line width=2.pt,color=zzttqq,fill=zzttqq,fill opacity=0.10000000149011612] (0.75,0.25) -- (0.75,0.375) -- (0.5,0.625) -- (0.5,0.4444444444) -- cycle;
\fill[line width=2.pt,color=zzttqq,fill=zzttqq,fill opacity=0.10000000149011612] (0.75,0.375) -- (0.75,0.5) -- (0.5,0.75) -- (0.5,0.625) -- cycle;
\fill[line width=2.pt,color=zzttqq,fill=zzttqq,fill opacity=0.10000000149011612] (0.75,0.) -- (0.75,-0.125) -- (0.5,-0.25) -- (0.5,0.) -- cycle;
\fill[line width=2.pt,color=zzttqq,fill=zzttqq,fill opacity=0.10000000149011612] (0.75,-0.125) -- (0.75,-0.25) -- (0.5,-0.5) -- (0.5,-0.25) -- cycle;
\fill[line width=2.pt,color=zzttqq,fill=zzttqq,fill opacity=0.10000000149011612] (0.75,-0.25) -- (0.75,-0.375) -- (0.5,-0.625) -- (0.5,-0.5) -- cycle;
\fill[line width=2.pt,color=zzttqq,fill=zzttqq,fill opacity=0.10000000149011612] (0.75,-0.375) -- (0.75,-0.5) -- (0.5,-0.75) -- (0.5,-0.625) -- cycle;
\fill[line width=2.pt,color=zzttqq,fill=zzttqq,fill opacity=0.10000000149011612] (-0.5,0.) -- (0.,0.) -- (0.,0.5) -- (-0.5,0.25) -- cycle;
\fill[line width=2.pt,color=zzttqq,fill=zzttqq,fill opacity=0.10000000149011612] (-0.5,0.25) -- (-0.5,0.4444444444) -- (0.,0.75) -- (0.,0.5) -- cycle;
\fill[line width=2.pt,color=zzttqq,fill=zzttqq,fill opacity=0.10000000149011612] (-0.5,0.4444444444) -- (-0.5,0.625) -- (0.,0.875) -- (0.,0.75) -- cycle;
\fill[line width=2.pt,color=zzttqq,fill=zzttqq,fill opacity=0.10000000149011612] (-0.5,0.625) -- (-0.5,0.75) -- (0.,0.9375) -- (0.,0.875) -- cycle;
\fill[line width=2.pt,color=zzttqq,fill=zzttqq,fill opacity=0.10000000149011612] (-0.5,0.) -- (-0.5,-0.25) -- (0.,-0.5) -- (0.,0.) -- cycle;
\fill[line width=2.pt,color=zzttqq,fill=zzttqq,fill opacity=0.10000000149011612] (-0.5,-0.25) -- (-0.5,-0.5) -- (0.,-0.75) -- (0.,-0.5) -- cycle;
\fill[line width=2.pt,color=zzttqq,fill=zzttqq,fill opacity=0.10000000149011612] (-0.5,-0.5) -- (-0.5,-0.625) -- (0.,-0.875) -- (0.,-0.75) -- cycle;
\fill[line width=2.pt,color=zzttqq,fill=zzttqq,fill opacity=0.10000000149011612] (-0.5,-0.625) -- (-0.5,-0.75) -- (0.,-0.9375) -- (0.,-0.875) -- cycle;
\fill[line width=2.pt,color=zzttqq,fill=zzttqq,fill opacity=0.10000000149011612] (-0.7499799968702604,0.) -- (-0.5,0.) -- (-0.5,0.25) -- (-0.75,0.125) -- cycle;
\fill[line width=2.pt,color=zzttqq,fill=zzttqq,fill opacity=0.10000000149011612] (-0.75,0.125) -- (-0.75,0.25) -- (-0.5,0.4444444444) -- (-0.5,0.25) -- cycle;
\fill[line width=2.pt,color=zzttqq,fill=zzttqq,fill opacity=0.10000000149011612] (-0.75,0.25) -- (-0.75,0.375) -- (-0.5,0.625) -- (-0.5,0.4444444444) -- cycle;
\fill[line width=2.pt,color=zzttqq,fill=zzttqq,fill opacity=0.10000000149011612] (-0.75,0.375) -- (-0.75,0.5) -- (-0.5,0.75) -- (-0.5,0.625) -- cycle;
\fill[line width=2.pt,color=zzttqq,fill=zzttqq,fill opacity=0.10000000149011612] (-0.7499799968702604,0.) -- (-0.75,-0.125) -- (-0.5,-0.25) -- (-0.5,0.) -- cycle;
\fill[line width=2.pt,color=zzttqq,fill=zzttqq,fill opacity=0.10000000149011612] (-0.75,-0.125) -- (-0.75,-0.25) -- (-0.5,-0.5) -- (-0.5,-0.25) -- cycle;
\fill[line width=2.pt,color=zzttqq,fill=zzttqq,fill opacity=0.10000000149011612] (-0.75,-0.25) -- (-0.75,-0.375) -- (-0.5,-0.625) -- (-0.5,-0.5) -- cycle;
\fill[line width=2.pt,color=zzttqq,fill=zzttqq,fill opacity=0.10000000149011612] (-0.75,-0.375) -- (-0.75,-0.5) -- (-0.5,-0.75) -- (-0.5,-0.625) -- cycle;
\draw [line width=0.8pt] (0.5,1.)-- (0.5,-1.);
\draw [line width=0.8pt] (-1.,1.)-- (1.,1.);
\draw [line width=0.8pt] (1.,1.)-- (1.,-1.);
\draw [line width=0.8pt] (1.,-1.)-- (-1.,-1.);
\draw [line width=0.8pt] (-1.,-1.)-- (-1.,1.);
\draw [line width=0.8pt] (0.75,1.)-- (0.75,-1.);
\draw [line width=2.pt,color=zzttqq] (0.,0.)-- (0.5,0.);
\draw [line width=2.pt,color=zzttqq] (0.5,0.)-- (0.5,0.25);
\draw [line width=2.pt,color=zzttqq] (0.5,0.25)-- (0.,0.5);
\draw [line width=2.pt,color=zzttqq] (0.,0.5)-- (0.,0.);
\draw [line width=2.pt,color=zzttqq] (0.,0.5)-- (0.5,0.25);
\draw [line width=2.pt,color=zzttqq] (0.5,0.25)-- (0.5,0.4444444444);
\draw [line width=2.pt,color=zzttqq] (0.5,0.4444444444)-- (0.,0.75);
\draw [line width=2.pt,color=zzttqq] (0.,0.75)-- (0.,0.5);
\draw [line width=2.pt,color=zzttqq] (0.,0.75)-- (0.5,0.4444444444);
\draw [line width=2.pt,color=zzttqq] (0.5,0.4444444444)-- (0.5,0.625);
\draw [line width=2.pt,color=zzttqq] (0.5,0.625)-- (0.,0.875);
\draw [line width=2.pt,color=zzttqq] (0.,0.875)-- (0.,0.75);
\draw [line width=2.pt,color=zzttqq] (0.,0.9375)-- (0.,0.875);
\draw [line width=2.pt,color=zzttqq] (0.,0.875)-- (0.5,0.625);
\draw [line width=2.pt,color=zzttqq] (0.5,0.625)-- (0.5,0.75);
\draw [line width=2.pt,color=zzttqq] (0.5,0.75)-- (0.,0.9375);
\draw [line width=2.pt,color=zzttqq] (0.,-0.5)-- (0.,-0.75);
\draw [line width=2.pt,color=zzttqq] (0.,-0.75)-- (0.5,-0.5);
\draw [line width=2.pt,color=zzttqq] (0.5,-0.5)-- (0.5,-0.25);
\draw [line width=2.pt,color=zzttqq] (0.5,-0.25)-- (0.,-0.5);
\draw [line width=2.pt,color=zzttqq] (0.,-0.75)-- (0.,-0.875);
\draw [line width=2.pt,color=zzttqq] (0.,-0.875)-- (0.5,-0.625);
\draw [line width=2.pt,color=zzttqq] (0.5,-0.625)-- (0.5,-0.5);
\draw [line width=2.pt,color=zzttqq] (0.,-0.875)-- (0.,-0.9375);
\draw [line width=2.pt,color=zzttqq] (0.,-0.9375)-- (0.5,-0.75);
\draw [line width=2.pt,color=zzttqq] (0.5,-0.75)-- (0.5,-0.625);
\draw [line width=2.pt,color=zzttqq] (0.5,-0.625)-- (0.,-0.875);
\draw [line width=2.pt,color=zzttqq] (0.,0.)-- (0.5,0.);
\draw [line width=2.pt,color=zzttqq] (0.5,0.)-- (0.5,-0.25);
\draw [line width=2.pt,color=zzttqq] (0.5,-0.25)-- (0.,-0.5);
\draw [line width=2.pt,color=zzttqq] (0.,-0.5)-- (0.,0.);
\draw [line width=2.pt,color=zzttqq] (0.5,0.)-- (0.75,0.);
\draw [line width=2.pt,color=zzttqq] (0.75,0.)-- (0.75,0.125);
\draw [line width=2.pt,color=zzttqq] (0.75,0.125)-- (0.5,0.25);
\draw [line width=2.pt,color=zzttqq] (0.5,0.25)-- (0.5,0.);
\draw [line width=2.pt,color=zzttqq] (0.5,0.25)-- (0.75,0.125);
\draw [line width=2.pt,color=zzttqq] (0.75,0.125)-- (0.75,0.25);
\draw [line width=2.pt,color=zzttqq] (0.75,0.25)-- (0.5,0.4444444444);
\draw [line width=2.pt,color=zzttqq] (0.5,0.4444444444)-- (0.5,0.25);
\draw [line width=2.pt,color=zzttqq] (0.75,0.25)-- (0.75,0.375);
\draw [line width=2.pt,color=zzttqq] (0.75,0.375)-- (0.5,0.625);
\draw [line width=2.pt,color=zzttqq] (0.5,0.625)-- (0.5,0.4444444444);
\draw [line width=2.pt,color=zzttqq] (0.5,0.4444444444)-- (0.75,0.25);
\draw [line width=2.pt,color=zzttqq] (0.75,0.375)-- (0.75,0.5);
\draw [line width=2.pt,color=zzttqq] (0.75,0.5)-- (0.5,0.75);
\draw [line width=2.pt,color=zzttqq] (0.5,0.75)-- (0.5,0.625);
\draw [line width=2.pt,color=zzttqq] (0.5,0.625)-- (0.75,0.375);
\draw [line width=2.pt,color=zzttqq] (0.75,0.)-- (0.75,-0.125);
\draw [line width=2.pt,color=zzttqq] (0.75,-0.125)-- (0.5,-0.25);
\draw [line width=2.pt,color=zzttqq] (0.5,-0.25)-- (0.5,0.);
\draw [line width=2.pt,color=zzttqq] (0.5,0.)-- (0.75,0.);
\draw [line width=2.pt,color=zzttqq] (0.75,-0.125)-- (0.75,-0.25);
\draw [line width=2.pt,color=zzttqq] (0.75,-0.25)-- (0.5,-0.5);
\draw [line width=2.pt,color=zzttqq] (0.5,-0.5)-- (0.5,-0.25);
\draw [line width=2.pt,color=zzttqq] (0.5,-0.25)-- (0.75,-0.125);
\draw [line width=2.pt,color=zzttqq] (0.75,-0.25)-- (0.75,-0.375);
\draw [line width=2.pt,color=zzttqq] (0.75,-0.375)-- (0.5,-0.625);
\draw [line width=2.pt,color=zzttqq] (0.5,-0.625)-- (0.5,-0.5);
\draw [line width=2.pt,color=zzttqq] (0.5,-0.5)-- (0.75,-0.25);
\draw [line width=2.pt,color=zzttqq] (0.75,-0.375)-- (0.75,-0.5);
\draw [line width=2.pt,color=zzttqq] (0.75,-0.5)-- (0.5,-0.75);
\draw [line width=2.pt,color=zzttqq] (0.5,-0.75)-- (0.5,-0.625);
\draw [line width=2.pt,color=zzttqq] (0.5,-0.625)-- (0.75,-0.375);
\draw [line width=0.8pt] (-0.5,1.)-- (-0.5,-1.);
\draw [line width=2.pt,color=zzttqq] (-0.5,0.)-- (0.,0.);
\draw [line width=2.pt,color=zzttqq] (0.,0.)-- (0.,0.5);
\draw [line width=2.pt,color=zzttqq] (0.,0.5)-- (-0.5,0.25);
\draw [line width=2.pt,color=zzttqq] (-0.5,0.25)-- (-0.5,0.);
\draw [line width=2.pt,color=zzttqq] (-0.5,0.25)-- (-0.5,0.4444444444);
\draw [line width=2.pt,color=zzttqq] (-0.5,0.4444444444)-- (0.,0.75);
\draw [line width=2.pt,color=zzttqq] (0.,0.75)-- (0.,0.5);
\draw [line width=2.pt,color=zzttqq] (-0.5,0.4444444444)-- (-0.5,0.625);
\draw [line width=2.pt,color=zzttqq] (-0.5,0.625)-- (0.,0.875);
\draw [line width=2.pt,color=zzttqq] (0.,0.875)-- (0.,0.75);
\draw [line width=2.pt,color=zzttqq] (0.,0.75)-- (-0.5,0.4444444444);
\draw [line width=2.pt,color=zzttqq] (-0.5,0.625)-- (-0.5,0.75);
\draw [line width=2.pt,color=zzttqq] (-0.5,0.75)-- (0.,0.9375);
\draw [line width=2.pt,color=zzttqq] (0.,0.9375)-- (0.,0.875);
\draw [line width=2.pt,color=zzttqq] (0.,0.875)-- (-0.5,0.625);
\draw [line width=2.pt,color=zzttqq] (-0.5,0.)-- (-0.5,-0.25);
\draw [line width=2.pt,color=zzttqq] (-0.5,-0.25)-- (0.,-0.5);
\draw [line width=2.pt,color=zzttqq] (0.,-0.5)-- (0.,0.);
\draw [line width=2.pt,color=zzttqq] (0.,0.)-- (-0.5,0.);
\draw [line width=2.pt,color=zzttqq] (-0.5,-0.25)-- (-0.5,-0.5);
\draw [line width=2.pt,color=zzttqq] (-0.5,-0.5)-- (0.,-0.75);
\draw [line width=2.pt,color=zzttqq] (0.,-0.75)-- (0.,-0.5);
\draw [line width=2.pt,color=zzttqq] (0.,-0.5)-- (-0.5,-0.25);
\draw [line width=2.pt,color=zzttqq] (-0.5,-0.5)-- (-0.5,-0.625);
\draw [line width=2.pt,color=zzttqq] (-0.5,-0.625)-- (0.,-0.875);
\draw [line width=2.pt,color=zzttqq] (0.,-0.875)-- (0.,-0.75);
\draw [line width=2.pt,color=zzttqq] (0.,-0.75)-- (-0.5,-0.5);
\draw [line width=2.pt,color=zzttqq] (-0.5,-0.625)-- (-0.5,-0.75);
\draw [line width=2.pt,color=zzttqq] (-0.5,-0.75)-- (0.,-0.9375);
\draw [line width=2.pt,color=zzttqq] (0.,-0.9375)-- (0.,-0.875);
\draw [line width=2.pt,color=zzttqq] (0.,-0.875)-- (-0.5,-0.625);
\draw [line width=0.8pt] (-0.75,1.)-- (-0.7499599937405208,-1.);
\draw [line width=2.pt,color=zzttqq] (-0.7499799968702604,0.)-- (-0.5,0.);
\draw [line width=2.pt,color=zzttqq] (-0.5,0.)-- (-0.5,0.25);
\draw [line width=2.pt,color=zzttqq] (-0.5,0.25)-- (-0.75,0.125);
\draw [line width=2.pt,color=zzttqq] (-0.75,0.125)-- (-0.7499799968702604,0.);
\draw [line width=2.pt,color=zzttqq] (-0.75,0.125)-- (-0.75,0.25);
\draw [line width=2.pt,color=zzttqq] (-0.75,0.25)-- (-0.5,0.4444444444);
\draw [line width=2.pt,color=zzttqq] (-0.5,0.4444444444)-- (-0.5,0.25);
\draw [line width=2.pt,color=zzttqq] (-0.5,0.25)-- (-0.75,0.125);
\draw [line width=2.pt,color=zzttqq] (-0.75,0.25)-- (-0.75,0.375);
\draw [line width=2.pt,color=zzttqq] (-0.75,0.375)-- (-0.5,0.625);
\draw [line width=2.pt,color=zzttqq] (-0.5,0.625)-- (-0.5,0.4444444444);
\draw [line width=2.pt,color=zzttqq] (-0.5,0.4444444444)-- (-0.75,0.25);
\draw [line width=2.pt,color=zzttqq] (-0.75,0.375)-- (-0.75,0.5);
\draw [line width=2.pt,color=zzttqq] (-0.75,0.5)-- (-0.5,0.75);
\draw [line width=2.pt,color=zzttqq] (-0.5,0.75)-- (-0.5,0.625);
\draw [line width=2.pt,color=zzttqq] (-0.5,0.625)-- (-0.75,0.375);
\draw [line width=2.pt,color=zzttqq] (-0.7499799968702604,0.)-- (-0.75,-0.125);
\draw [line width=2.pt,color=zzttqq] (-0.75,-0.125)-- (-0.5,-0.25);
\draw [line width=2.pt,color=zzttqq] (-0.5,-0.25)-- (-0.5,0.);
\draw [line width=2.pt,color=zzttqq] (-0.5,0.)-- (-0.7499799968702604,0.);
\draw [line width=2.pt,color=zzttqq] (-0.75,-0.125)-- (-0.75,-0.25);
\draw [line width=2.pt,color=zzttqq] (-0.75,-0.25)-- (-0.5,-0.5);
\draw [line width=2.pt,color=zzttqq] (-0.5,-0.5)-- (-0.5,-0.25);
\draw [line width=2.pt,color=zzttqq] (-0.5,-0.25)-- (-0.75,-0.125);
\draw [line width=2.pt,color=zzttqq] (-0.75,-0.25)-- (-0.75,-0.375);
\draw [line width=2.pt,color=zzttqq] (-0.75,-0.375)-- (-0.5,-0.625);
\draw [line width=2.pt,color=zzttqq] (-0.5,-0.625)-- (-0.5,-0.5);
\draw [line width=2.pt,color=zzttqq] (-0.5,-0.5)-- (-0.75,-0.25);
\draw [line width=2.pt,color=zzttqq] (-0.75,-0.375)-- (-0.75,-0.5);
\draw [line width=2.pt,color=zzttqq] (-0.75,-0.5)-- (-0.5,-0.75);
\draw [line width=2.pt,color=zzttqq] (-0.5,-0.75)-- (-0.5,-0.625);
\draw [line width=2.pt,color=zzttqq] (-0.5,-0.625)-- (-0.75,-0.375);
\draw [line width=0.8pt] (0.,1.)-- (0.,-1.);
\begin{scriptsize}
\draw [fill=ududff] (-1.,1.) circle (1.5pt);
\draw[color=ududff] (-0.9231368192219812,1.0498679065974856) node {(-1,1)};
\draw [fill=ududff] (1.,1.) circle (1.5pt);
\draw[color=ududff] (1.0649813575134874,1.0498679065974856) node {(1,1)};
\draw [fill=ududff] (1.,-1.) circle (1.5pt);
\draw[color=ududff] (1.0754635271623563,-0.9487324397867021) node {(1,-1)};
\draw [fill=ududff] (-1.,-1.) circle (1.5pt);
\draw[color=ududff] (-0.9126546495731124,-0.9487324397867021) node {(-1,-1)};
\draw [fill=black] (0.,0.) circle (1.0pt);
\draw[color=black] (0.06568118432131853,0.04357962030614645) node {(0,0)};
\draw [fill=black] (0.,0.5) circle (1.0pt);
\draw [fill=black] (0.5,0.) circle (1.0pt);
\draw [fill=black] (0.,0.75) circle (1.0pt);
\draw [fill=black] (0.,0.875) circle (1.0pt);
\draw [fill=black] (0.,0.9375) circle (1.0pt);
\draw [fill=black] (0.,-0.75) circle (1.0pt);
\draw [fill=black] (0.,-0.875) circle (1.0pt);
\draw [fill=black] (0.,-0.9375) circle (1.0pt);
\draw [fill=black] (0.5,-0.25) circle (1.0pt);
\draw [fill=black] (0.75,0.) circle (1.0pt);
\draw [fill=black] (0.75,0.25) circle (1.0pt);
\draw [fill=black] (0.75,0.5) circle (1.0pt);
\draw [fill=black] (0.75,-0.25) circle (1.0pt);
\draw [fill=black] (0.75,-0.5) circle (1.0pt);
\draw [fill=black] (0.5,0.25) circle (1.0pt);
\draw[color=zzttqq] (0.28056566212313105,0.22701758916133852) node {$T_{0,0}$};
\draw [fill=black] (0.5,0.4444444444) circle (1.0pt);
\draw[color=zzttqq] (0.2665894359246392,0.5100361696807777) node {$T_{0,1}$};
\draw [fill=black] (0.5,0.625) circle (1.0pt);
\draw [fill=black] (0.5,0.75) circle (1.0pt);
\draw [fill=black] (0.5,-0.5) circle (1.0pt);
\draw [fill=black] (0.,-0.5) circle (1.0pt);
\draw[color=zzttqq] (0.28755377522237696,-0.504613115511143346) node {$T_{0,-2}$};
\draw [fill=black] (0.5,-0.625) circle (1.0pt);
\draw [fill=black] (0.5,-0.75) circle (1.0pt);
\draw[color=zzttqq] (0.28755377522237696,-0.14684646164829102) node {$T_{0,-1}$};
\draw [fill=black] (0.75,0.125) circle (1.0pt);
\draw[color=zzttqq] (0.6579237694824115,0.113267806232152546) node {$T_{1,0}$};
\draw [fill=black] (0.75,0.375) circle (1.0pt);
\draw [fill=black] (0.75,-0.125) circle (1.0pt);
\draw[color=zzttqq] (0.65649118825816575,-0.05250693480847798) node {$T_{1,-1}$};
\draw [fill=black] (0.75,-0.375) circle (1.0pt);
\draw [fill=uuuuuu] (-0.5,0.) circle (1.0pt);
\draw [fill=uuuuuu] (-0.5,0.4444444444) circle (1.0pt);
\draw [fill=uuuuuu] (-0.5,0.75) circle (1.0pt);
\draw [fill=uuuuuu] (-0.5,0.25) circle (1.0pt);
\draw[color=zzttqq] (-0.21209631137370738,0.22701758916133852) node {$T_{-1,0}$};
\draw[color=zzttqq] (-0.21209631137370738,0.51275064524288913) node {$T_{-1,1}$};
\draw [fill=uuuuuu] (-0.5,0.625) circle (1.0pt);
\draw [fill=uuuuuu] (-0.5,-0.5) circle (1.0pt);
\draw [fill=uuuuuu] (-0.5,-0.75) circle (1.0pt);
\draw [fill=uuuuuu] (-0.5,-0.25) circle (1.0pt);
\draw[color=zzttqq] (-0.20510819827446145,-0.14684646164829102) node {$T_{-1,-1}$};
\draw[color=zzttqq] (-0.20510819827446145,-0.504613115511143346) node {$T_{-1,-2}$};
\draw [fill=uuuuuu] (-0.5,-0.625) circle (1.0pt);
\draw [fill=uuuuuu] (-0.7499799968702604,0.) circle (1.0pt);
\draw [fill=uuuuuu] (-0.75,0.25) circle (1.0pt);
\draw [fill=uuuuuu] (-0.75,0.5) circle (1.0pt);
\draw [fill=uuuuuu] (-0.75,0.125) circle (1.0pt);
\draw[color=zzttqq] (-0.625859603621833649,0.103267806232152546) node {$T_{-2,0}$};
\draw [fill=uuuuuu] (-0.75,0.375) circle (1.0pt);
\draw [fill=uuuuuu] (-0.75,-0.25) circle (1.0pt);
\draw [fill=uuuuuu] (-0.75,-0.5) circle (1.0pt);
\draw [fill=uuuuuu] (-0.75,-0.125) circle (1.0pt);
\draw[color=zzttqq] (-0.625789722490841189,-0.056250693480847798) node {$T_{-2,-1}$};
\draw [fill=uuuuuu] (-0.75,-0.375) circle (1.0pt);
\draw[color=black] (0.875,0) node {$\ldots$};
\draw[color=black] (-0.875,0) node {$\ldots$};
\end{scriptsize}
\end{tikzpicture}
\caption{Decomposing $J$ into the trapezoids $T_{a,b}$}\label{F:Trapezoids}
\end{figure}

Consider the map $p \colon J\to J $ which fixes the boundary $\partial J$ pointwise and which maps, for every $a,b\in\mathbb{Z}^2$, the trapezoid  $T_{a,b}$ to $T_{a,b+1}$ via the   bilinear interpolation  which sends the vertices 
 $v_{a,b}^{0,0},  v_{a,b}^{0,1},  v_{a,b}^{1,0},  v_{a,b}^{1,1}$ of $T_{a,b}$ to the corresponding   $v_{a,b+1}^{0,0},  v_{a,b+1}^{0,1},  v_{a,b+1}^{1,0},  v_{a,b+1}^{1,1}$ of $T_{a,b+1}$. That is,  $p\upharpoonright T_{a,b}=\chi_{a,b+1}\circ  \chi^{-1}_{a,b}$, where for all $a,b\in\mathbb{Z}$ the map $ \chi_{a,b}\colon [0,1]^2 \to T_{a,b}$ is given by
 \[\chi_{a,b}(s,t):=(1-s)(1-t)v^{0,0}_{a,b}  +(1-s)t v^{0,1}_{a,b}  +s(1-t) v^{1,0}_{a,b} +s t v^{1,1}_{a,b}. \]
 It is clear that $p$ is a self-homeomorphism of the interior $J\setminus \partial J$ of $J$. The fact that $p$ is continuous on $\partial J$, and hence a self-homeomorphism of $J$, follows since $\lim_{k\to \infty} s^k_n= 0$ holds for all $n\in\mathbb{N}$. Let now $\mathcal{B}_0$ be any deballing recipe of $\bigcup_{a\in\mathbb{Z}}T_{a,b}$. The desired deballing recipe  $\mathcal{B}$ of $J$ and the homeomorphism $\pi\colon \mathcal{B}(J)\to \mathcal{B}(J)$ are now simply given by:
\[\mathcal{B}:=\bigcup_{b\in\mathbb{Z}}(p^b)^{*}\mathcal{B}_0 \quad \text{ and } \quad \pi:= p\res  \mathcal{B}(J),\]
where $(p^b)^{*}\mathcal{B}_0$ is the pull back of $\mathcal{B}_0$ under the $b$-iterate $p^b$ of $p$; see Section \ref{SS:Preliminaries}.
\end{proof}

\section{Proof of Lemma \ref{L:main2}}\label{SL:main2}
Recall that $L:=[0,1]\times\mathbb{R}$, $\partial L=\{0,1\} \times \mathbb{R}$, and for every   $k,\ell\in \mathbb{Z}$, we have
\[
h_{k,\ell}^{\partial}\colon \partial L\to \partial L \quad \text{ given by } \quad h^{\partial}_{k,\ell}(0,y)=(0,y+k) \text{ and } h^{\partial}_{k,\ell}(1,y)=(1,y+\ell).
\]
For every $y\in\mathbb{R}$, let  $\lfloor y \rfloor$  be the unique integer with  $0\leq y-\lfloor y \rfloor <1$. We will define 
deballing recipe $\mathcal{B}$ for $L$ so that,  for every $k,\ell\in\mathbb{Z}$, there exists a homeomorphism $h^{\mathcal{B}}_{k,\ell}\colon \mathcal{B}(L) \to  \mathcal{B}(L)$ which  extends $h^{\partial}_{k,\ell}$,
so that for all $(x,y)\in L$, setting $(\tilde{x},\tilde{y}):=h^{\mathcal{B}}_{k,\ell}(x,y)$, we have that
\begin{equation}
 \lfloor y \rfloor +\min\{k,\ell\}-1 \;  \leq \;  \tilde{y} \;  \leq  \; (\lfloor y \rfloor+1) +\max\{k,\ell\}+1.  
\end{equation}

\begin{proof}[Proof of Lemma \ref{L:main2}]
Let $\{(a_n,b_n)\in\mathbb{N}\}$ be a enumeration of all pairs of integers  $a,b\in \mathbb{Z}$. For each $n\in\mathbb{N}$ let $p_n\colon [0,1]\to \mathbb{R}$ be the function whose graph $\{(t,p_n(t))\colon t\in [0,1]\}$ is the straight path from $(0,a_n)$ to $(1,b_n)$.  Namely, $p_{n}(t)=(1-t)a_n+tb_n$. 
Notice that the union of the graphs of the maps $p_n$ is dense in $L$. Hence, there is no deballing recipe for $L$ whose elements avoid this union. Our first goal is to define,  small  perturbations $\pi_n$ of $p_n$, together with
the  deballing recipe $\mathcal{B}$ for $L$ so that each $B\in\mathcal{B}$ avoids the graphs of the pertubations.

Let $Q=\{(x_n,y_n)\in\mathbb{N}\}$ be a countable dense subset of $L\setminus \partial L$ so that $Q$  does not intersect the graph of any function $p_n$, that is,    $p_n(x_m)\neq y_m$ for all $n,m\in\mathbb{N}$. We can always find such $Q$, since the graph of each $p_n$ is nowhere dense in $L$.

\begin{claim}\label{C:1} For every $n\in \mathbb{N}$ one may define:
\begin{enumerate}
\item a continuous function $\pi_n\colon [0,1] \to \mathbb{R}$ with  $\pi_n(0)=a_n$ and $\pi_n(1)=b_n$ so that:
\begin{enumerate}
\item for all $t\in [0,1]$ we have $\min\{a_n,b_n\}-1\leq  \pi_n(t) \leq \max\{a_n,b_n\}+1$; 
\item  for all $t\in [0,1]$ and all $m,n$ with $b_n-a_n=b_m-a_m$ we have $\pi_n(t)\neq \pi_m(t)$;
\item  for all $t\in [0,1]$ we have $(t,\pi_n(t))\not\in Q$.
\end{enumerate}
\item a closed ball $B_n=\{(x,y)\in L \colon |x-x_n|,|y-y_n|\leq r_n\}$, in the $\ell_\infty$--norm,
with center $(x_n,y_n)$ and  some radius $r_n<1/2$ which will be specified later, so that:
\begin{enumerate}
\item  $B_n\cap \partial L=\emptyset$ and $B_n\cap B_m=\emptyset$ for all $n,m\in \mathbb{N}$ with $n\neq m$;
\item $B_n \cap \{\big(t,\pi_m(t)\big)\colon t\in [0,1]\}=\emptyset$  for all  $n,m\in \mathbb{N}$.
\end{enumerate}
\end{enumerate}
\end{claim}
\begin{proof}[Proof of Claim]
 Assume that $(\pi_m\colon m<n)$ and $(B_m\colon m<n)$ have been defined satisfying the above. We  first define $\pi_n$ and then $B_n$. Let  $d_{\infty}$ be the metric coming form the  $\ell_{\infty}$--norm.

 Let  $q\colon [0,1]\to [\min\{a_n,b_n\}-1,\max\{a_n,b_n\}+1]$ be any continuous function with $q(0)=a_n$ and $q(1)=b_n$ so that    for all $t\in [0,1]$ and all $m<n$ with $b_n-a_n=b_m-a_m$ we have  that $q(t)\neq \pi_m(t)$. Such function exists since  1(a) and 1(b) hold for  $(\pi_m\colon m<n)$. Moreover, we can take $q$ to be piece-wise linear. Finally, since $Q$ is countable, we can always adjust the slopes of each linear piece of $q$ to make sure that we additionally have $(t,q(t))\not\in Q$ for all $t\in[0,1]$.
  The plan is to define $\pi_n$ so that it coincides with $q$ until its graph reaches close enough to some $B_m$ with $m<n$ which intersects the graph of $q$. We will then momentarily divert $\pi_n$ from $q$ in order to circumvent  $B_m$, but without violating 1(a),(b),(c) and 2(b) above. We then continue by letting $\pi_n$ coincide  with $q$  again until the next ball from  $(B_m\colon m<n)$ gets in the way, and has to be  circumvented in a similar fashion.

To circumvent $B_m$ as  in the previous paragraph we work as follows.  First  fix  some $\varepsilon>0$  so that  for all $i,j<n$ we have that   $\varepsilon <  d_{\infty}(B_i,B_{j})/2$. By shrinking $\varepsilon$ even further we may assume that for all $(x,y)$ with $|x-x_m|,|y-y_m|<r_m+\varepsilon$ and every   $i<n$ with $b_n-a_n=b_i-a_i$ we have that $\pi_m(x)\neq y$. The fact that we can find such $\varepsilon$ is a consequence of 2(b) and 1(b) respectively, since  the balls in
 in $(B_m\colon m<n)$ and the graphs of $(\pi_m \colon m<n)$ are compact.  Finally   since $r_n<1/2$ and the graph of $q\colon [0,1]\to [\min\{a_n,b_n\}-1,\max\{a_n,b_n\}+1]$  meets $B_m$, we can shrink $\varepsilon$ even further if necessary, so that  either the interval $(y_m+r_m,y_m+r_m+\varepsilon)$ or the interval $(y_m-r_m-\varepsilon,y_m-r_m)$  is entirely contained in $[\min\{a_n,b_n\}-1,\max\{a_n,b_n\}+1]$.  Without loss of generality we may assume that the former is the case:
\[ (y_m+r_m,y_m+r_m+\varepsilon)\subseteq [\min\{a_n,b_n\}-1,\max\{a_n,b_n\}+1]\]
 If the latter is the case, the rest of the argument is similar.

 Assume now that we have defined $\pi_n$ so far on the domain $[0,x_m-r_m-\varepsilon/2]$ and that $\pi_n(x_m-r_m-\varepsilon/2)=q(x_m-r_m-\varepsilon/2)$. We will extend the definition of  $\pi_n$  on the  domain   $[x_m-r_m-\varepsilon/2,x_m+r_m+\varepsilon/2]$ as follows. Since $Q$ is countable we may find some $y_0,y_1\in (y_m+r_m,y_m+r_m+\varepsilon)$ so that  the graph of the function
\[ r(t)=  \begin{cases} \frac{y_0-q(x_m-r_m-\varepsilon/2) }{\varepsilon/2}t+q(x_m-r_m-\varepsilon/2) & \text{ if }  t\in [x_m-r_m-\varepsilon/2,x_m-r_m]\\
\frac{y_1-y_0}{2r_m}t+y_0  & \text{ if }  t\in [x_m-r_m,x_m+r_m]\\
\frac{q(x_m+r_m+\varepsilon/2)-y_1 }{\varepsilon/2}t+q(x_m+r_m+\varepsilon/2) & \text{ if }  t\in [x_m+r_m,x_m+r_m+\varepsilon/2]\\
  \end{cases}\] 
  does not intersect $Q$.  We may now extend the definition of $\pi_n$ to agree with $r$ on the interval  $[x_m-r_m-\varepsilon/2,x_m+r_m+\varepsilon/2]$ and then again with $q$ until the  next ball from  $(B_m\colon m<n)$ needs to be circumvented or  $(1,b_n)$ has been reached. This concludes the definition of $\pi_n$.

\begin{figure}[ht!]
\begin{center}
\begin{tikzpicture}
[line cap=round,line join=round,>=triangle 45,x=1cm,y=1cm,scale=0.8]
\fill[line width=2pt,color=zzttqq,fill=zzttqq,fill opacity=0.46] (-2,-2) -- (2,-2) -- (2,2) -- (-2,2) -- cycle;
\fill[line width=0.4pt,color=zzttqq,fill=zzttqq,fill opacity=0.10000000149011612] (-3,-3) -- (3,-3) -- (3,3) -- (-3,3) -- cycle;
\draw [line width=1pt,color=zzttqq] (-2,-2)-- (2,-2);
\draw [line width=1pt,color=zzttqq] (2,-2)-- (2,2);
\draw [line width=1pt,color=zzttqq] (2,2)-- (-2,2);
\draw [line width=1pt,color=zzttqq] (-2,2)-- (-2,-2);
\draw [line width=0.2pt,color=zzttqq] (-3,-3)-- (3,-3);
\draw [line width=0.2pt,color=zzttqq] (3,-3)-- (3,3);
\draw [line width=0.2pt,color=zzttqq] (3,3)-- (-3,3);
\draw [line width=0.2pt,color=zzttqq] (-3,3)-- (-3,-3);
\draw [line width=1.4pt] (-2,-1)-- (-4,0);
\draw [line width=1.4pt] (-4,0)-- (-6,-2);
\draw [line width=1.4pt] (-2,-1)-- (0,0);
\draw [line width=1.4pt] (0,0)-- (2,-1);
\draw [line width=1.4pt] (2,-1)-- (4,0);
\draw [line width=1.4pt] (4,0)-- (6,2);
\draw [line width=1.4pt,dash pattern=on 1pt off 2pt] (-2.5,-0.75)-- (-2,2.8);
\draw [line width=1.4pt,dash pattern=on 1pt off 2pt] (-2,2.8)-- (2,2.375);
\draw [line width=1.4pt,dash pattern=on 1pt off 2pt] (2,2.375)-- (2.5,-0.75);
\draw [line width=1pt,dotted] (2,-2)-- (5,-2);
\draw [line width=1pt,dotted] (2,-3)-- (5,-3);
\draw (5.2,-2.5)  node {$\varepsilon$};
\draw [fill=black] (-2,2.8) circle (2.5pt); 
\draw (-2.4,2.8)  node {$y_0$};
\draw [fill=black] (2,2.375) circle (2.5pt); 
\draw (2.4,2.375)  node {$y_1$};
\draw [fill=black] (-2.5,-0.75) circle (2.5pt);
\draw [fill=black] (2.5,-0.75) circle (2.5pt);
\draw (0,-1.5)  node {$B_m$};
\end{tikzpicture}      
\end{center}
\caption{Modifying  $q$ to get a $\pi_n$ which avoids $B_m$}
\end{figure}
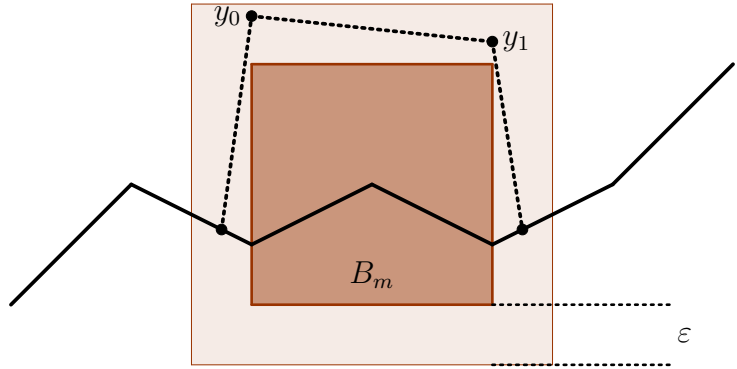

 The definition of $B_n$ is simple. Let $K$ be the union of all the graphs  
 of the maps from $(\pi_m\colon m\leq n)$ and of the all the balls from   $(B_m\colon m< n)$. Since $K$ is compact and it does not intersect $Q$ we can find some $r_n$ small enough so that the $d_{\infty}$-ball with center $(x_n,y_n)$ and radius $r_n$ does not intersect $K$. Since $x_n\in(0,1)$ we can further shrink $r_n$ to ensure that  this ball does intersect $\partial L$ either.
\end{proof}

\begin{claim}\label{C:2}
In the statement of previous claim,  we may additionally assume without loss of generality that  $\pi_n$ is the constant map ($\pi_n=p_n$) for all $n\in\mathbb{N}$ with $a_n=b_n$.
\end{claim}
\begin{proof}
 This is a matter of minor adjustments of the proof of Claim \ref{C:1} by including some additional bookkeeping.  The main observation is that if $S$ is the set of all $m\in \mathbb{N}$ with $a_n=b_n$, then the union of the graphs of the functions $\{p_m\colon m\in S\}$ is nowhere dense in $L$. Hence, for each $n\in\mathbb{N}$ we can  choose  $r_n$ small enough so that $B_n$ avoids the graph of $p_m$ for all $m\in S$. That is, there is no need to perturb $p_m$ to some new $\pi_m$ when $m\in S$.
\end{proof}

Let now $\mathcal{B}:=\{B_n\colon n\in\mathbb{N}\}$ be the collection of all  balls from the claim. Clearly $\mathcal{B}$ is a deballing recipe for $L$. Fix some $k,\ell\in\mathbb{Z}$. We will define  a  homeomorphism $h^{\mathcal{B}}_{k,\ell}\colon \mathcal{B}(L) \to  \mathcal{B}(L)$ satisfying the conclusion of the lemma. For every  $a,b\in\mathbb{Z}$ let $\pi_{a,b}$ denote the unique map from the claim  with $\pi_{a,b}(0)=a$ and $\pi_{a,b}(1)=b$. Let
\begin{eqnarray}
    Q_a:=\{(x,y)\in L\colon  \pi_{a,a}(x) \leq y \leq \pi_{a+1,a+1}(x) \}\label{Eq:1_fin} \\
    R_a:=\{(x,y)\in L\colon  \pi_{a+k,a+\ell}(x) \leq y \leq \pi_{a+k+1,a+\ell+1}(x) \} \label{Eq:2_fin}.
\end{eqnarray}

By  Property 1(b) of Claim \ref{C:1} we have that  $Q_a$ is a two--dimensional ball whose boundary is the union of $\{(0,y)\colon a\leq y\leq a+1 \} \cup \{(1,y)\colon a\leq y\leq a+1 \} $ with the graphs of $\pi_{a,a}$ and $\pi_{a+1,a+1}$ and similarly for $R_a$. By Properties 2(a) and 2(b) of the claim we also have that $\mathcal{B}(Q_a)$ and $\mathcal{B}(R_a)$ are homeomorphic to the Sierpinski carpet. By Theorem \ref{T:whyburn}  we have a homeomorphism $h_a\colon \mathcal{B}(Q_a) \to \mathcal{B}(R_a)$ which extends the map $\partial h_a \colon \partial \mathcal{B}(Q_a) \to \partial \mathcal{B}(R_a)$ with 
\[\partial h_a((x,y)) = \begin{cases} 
 (x,y+k) \; & \; \text{ if } x=0\\
 (x,y+\ell) \; & \; \text{ if } x=1\\
 (x,\pi_{a+k,a+\ell}) \; & \; \text{ if }    y=\pi_{a,a}(x)\\
 (x,\pi_{a+k+1,a+\ell+1}) \; & \; \text{ if }    y=\pi_{a+1,a+1}(x)\\
\end{cases}\]
Let finally $h^{\mathcal{B}}_{k,\ell}\colon \mathcal{B}(L) \to  \mathcal{B}(L)$ be given by setting $h^{\mathcal{B}}_{k,\ell}(x,y)=h_a(x,y)$ if $(x,y)\in Q_a$.

Clearly $h^{\mathcal{B}}_{k,\ell}$ extends $h^{\partial}_{k,\ell}$.  Moreover, if $(x,y)\in Q_a$ and $k,\ell\in \mathbb{Z}$ then since,
$(\tilde{x} ,\tilde{y}):= h^{\mathcal{B}}_{k,\ell}(x,y)\in R_a$,
by combining (\ref{Eq:2_fin})  above with Property (1)(a) from Claim \ref{C:1}  we have that:
\[\min\{a+k,a+\ell\}-1 \leq \tilde{y}\leq \max\{a+k+1,a+\ell+1\}+1.\]
But by Claim \ref{C:2} we may assume  that $a=\lfloor y \rfloor$, which completes the proof.
\end{proof}

\bibliography{References}

\begin{thebibliography}{CDM05}

\bibitem[BCV24]{Basso}
G.~Basso, A.~Codenotti, and A.~Vaccaro.
\newblock Surfaces and other {P}eano {C}ontinua with no {G}eneric {C}hains.
\newblock {\em arXiv:2403.08667}, 2024.

\bibitem[Bre66]{Brechner}
B.L. Brechner.
\newblock On the {D}imensions of {C}ertain {S}paces of {H}omeomorphisms.
\newblock {\em Transactions of the American Mathematical Society},
  121(2):516--548, 1966.

\bibitem[CDM05]{CamDM}
R.~Camerlo, U.B. Darji, and A.~Marcone.
\newblock Classification {P}roblems in {C}ontinuum {T}heory.
\newblock {\em Transactions of the American Mathematical Society},
  357(11):4301--4328, 2005.

\bibitem[CG01]{CamGao}
R.~Camerlo and S.~Gao.
\newblock The completeness of the isomorphism relation for countable {B}oolean
  algebras.
\newblock {\em Transactions of the American Mathematical Society},
  353(2):491--518, 2001.

\bibitem[DV23]{Dudak}
J.~Dud\'ak and B.~Vejnar.
\newblock The complexity of homeomorphism relations on some classes of compacta
  with bounded topological dimension.
\newblock {\em Fundamenta Mathematicae}, 263:1--22, 2023.

\bibitem[Gao08]{Gao}
S.~Gao.
\newblock {\em Invariant {D}escriptive {S}et {T}heory}.
\newblock Chapman and Hall/CRC Press, 2008.

\bibitem[Hjo00]{Hjorth2010}
G.~Hjorth.
\newblock {\em Classification and {O}rbit {E}quivalence {R}elations}, volume~75
  of {\em {M}athematical {S}urveys and {M}onographs}.
\newblock American Mathematical Society, 2000.

\bibitem[KV20]{Kru}
P.~Krupski and B.~Vejnar.
\newblock The complexity of homeomorphism relations on some classes of
  compacta.
\newblock {\em The Journal of Symbolic Logic}, 85(2):733--–748, 2020.

\bibitem[Why58]{Whyburn1958}
G.~Whyburn.
\newblock Topological characterization of the {S}ierpiński curve.
\newblock {\em Fundamenta Mathematicae}, 45(1):320--324, 1958.

\bibitem[Zie16]{Zielinski}
J.~Zielinski.
\newblock The complexity of the homeomorphism relation between compact metric
  spaces.
\newblock {\em Advances in Mathematics}, 291:635--645, 2016.

\end{thebibliography}
\bibliographystyle{alpha}
\end{document}